\documentclass{amsart}
\usepackage{amsmath}
\usepackage{ amssymb}

\usepackage{fullpage}

\usepackage{graphicx}

\usepackage{xcolor}

\usepackage{tikz}
\usetikzlibrary{patterns}
\usetikzlibrary{arrows}
\definecolor{red}{rgb}{1,0,0}
\definecolor{blue}{rgb}{0,0,1}

\usepackage{amsthm}

\newtheorem{theorem}{Theorem}
\newtheorem{corollary}{Corollary}
\newtheorem{remark}{Remark}

\numberwithin{equation}{section}

\newcommand{\dx}{\partial_x}
\newcommand{\dz}{\partial_z}
\newcommand{\dy}{\partial_y}
\newcommand{\tx}{\tilde{x}}
\newcommand{\tz}{\tilde{z}}
\newcommand{\ttt}{\tilde{t}}
\newcommand{\teta}{\tilde{\eta}}

\newcommand{\tPhi}{\tilde{\Phi}}
\newcommand{\tP}{\tilde{P}}

\newcommand{\dsp}{{\displaystyle}}

\newcommand{\R}{{\mathbb{R}}}
\newcommand{\scrO}{{\mathcal{O}}}

\DeclareMathOperator*{\esssup}{ess\,sup}

\begin{document}

\title[Mechanical balance laws for two-dimensional Boussinesq systems]{Mechanical balance laws for two-dimensional Boussinesq systems}



\author{Chourouk El Hassanieh}
\address{\textbf{C. E.~Hassanieh:} Department of Mathematics, Faculty of Sciences 1,
Lebanese University, Beirut, Lebanon \newline
 $\&$ Sorbonne Universit\'{e}, 4 place Jussieu, 75005 Paris, France $\&$ Inria Paris}
\email{chourouk.el-hassanieh@inria.fr}

\author{Samer Israwi}
\address{\textbf{S.~Israwi:} Department of Mathematics, Faculty of Sciences 1,
Lebanese University, Beirut, Lebanon \newline $\&$ Arts, Sciences, and Technology University in Lebanon, CRAMS Center of Research in Applied Mathematics and Statistics, faculty of sciences, Beirut, Lebanon
}
\email{s\_israwi83@hotmail.com}

\author{Henrik Kalisch}
\address{\textbf{H.~Kalisch:} Department of Mathematics, University of Bergen, Postbox 7800, 5020 Bergen, Norway}
\email{henrik.kalisch@uib.no}

\author{Dimitrios Mitsotakis}
\address{\textbf{D.~Mitsotakis:} Victoria University of Wellington, School of Mathematics and Statistics, PO Box 600, Wellington 6140, New Zealand}
\email{dimitrios.mitsotakis@vuw.ac.nz}

\author{Amutha Senthilkumar}
\address{\textbf{A.~Senthilkumar:} Department of Mathematics, University of Bergen, Postbox 7800, 5020 Bergen, Norway}
\email{amutha.senthilkumar@math.uib.no}



\subjclass[2000]{35Q35, 76B15, 35B40}

\date{\today}


\keywords{Boussinesq systems, balance laws, asymptotic relations}

\begin{abstract}
Most of the asymptotically derived Boussinesq systems of water wave theory for long waves of small amplitude fail to satisfy exact mechanical conservation laws for mass, momentum and energy. It is thus only fair to consider approximate conservation laws that hold in the context of these systems. Although such approximate mass, momentum and energy conservation laws can be derived, the question of a rigorous mathematical justification still remains unanswered. The aim of this paper is to justify the formally derived mechanical balance laws for weakly nonlinear and weakly dispersive water wave Boussinesq systems. In particular, two asymptotic expansions used for the formal and rigorous derivation of the Boussinesq systems and the same are employed for the derivation and rigorous justification of the balance laws. Numerical validation of the asymptotic orders of approximation is also presented.
\end{abstract}

\maketitle

\section{Introduction}\label{sec:introduction}

The propagation of surface water waves is described by the Euler equations of fluid mechanics, which are accompanied by dynamic boundary conditions at the free surface and the sea floor, as detailed in \cite{Whitham1974}. The solutions of the Euler equations also satisfy essential conservation equations, such as those for mass, momentum, and energy. While the Euler equations are a well-justified model both physically and mathematically, as indicated in \cite{Lannes13}, solving them, whether theoretically or numerically, remains a formidable challenge. This difficulty arises from the fact that the domain in which dependent variables like velocity and pressure are defined changes over time and is bounded by the unknown free surface.

As a result, numerous approximate models have been developed to approximate solutions of the Euler equations. These approximations rely on simplification assumptions about the characteristics of the waves to be described. These assumptions typically give rise to what are known as {\em wave regimes}. Two significant examples of water wave regimes are the small amplitude and long wave regime and the large amplitude, long wave regime. In both of these regimes, the primary focus is on long water waves, which means that the waves are considered to have a much greater wavelength compared to the depth of the water. In the small amplitude regime, the waves are assumed to have small amplitudes in comparison to the water depth, whereas in the large amplitude regime, there is no restriction on the amplitude of the waves.

In both regimes, formal approximations of the Euler equations are typically referred to as Boussinesq systems, with the exception of the large amplitude and long wave regime, where these systems are often termed Serre-Green-Naghdi equations, as mentioned in \cite{Lannes13}. This nomenclature stems from the fact that the initial approximations were first formulated by J. Boussinesq \cite{Bous} under the small amplitude assumption. Subsequently, the small amplitude assumption was removed by F. Serre \cite{Serre1953} and independently by Green and Naghdi in \cite{GN1976}, who derived equations describing strongly nonlinear and weakly dispersive water waves. For more comprehensive details, please refer to \cite{Lannes13}.

The initial approximations of the Euler equations had several theoretical and sometimes practical limitations. Therefore, research efforts were focused on deriving mathematically sound Boussinesq systems that possess favorable nonlinear and dispersive properties. Building upon Boussinesq's work, a whole class of Boussinesq-type systems were derived \cite{BCS2002, BCL2015} and were subsequently justified both theoretically and numerically, as documented in \cite{Lannes13, DM2008, DMS2007, dms2009, DMS2010, IKKM, KLIG}. Since these models find applications in the nearshore zone, it is imperative to be able to assess their accuracy concerning mass, momentum, and energy balance laws.

In the one-dimensional case of the Serre equations, balance laws have been estimated using asymptotic techniques, and it has been established that these balance laws are satisfied by the solutions of the Serre equations with the same level of accuracy as they approximate the Euler equations, as discussed in \cite{KKM}. Similarly, for one-dimensional Boussinesq systems, similar estimates were obtained, as documented in \cite{AK2}.

In this paper, we focus on a two-dimensional shallow water wave regime modeled by a set of Boussinesq-type equations initially presented in \cite{BCL2015}. These two-dimensional Boussinesq Equations can be derived from the Water Waves problem through a straightforward asymptotic expansion based on the potential $\Phi$. They describe the propagation of long waves with small amplitudes on the surface of an ideal fluid. The motion of the free surface and the evolution of the velocity field of an incompressible, inviscid, and irrotational fluid under the influence of gravity is the actual notion of water waves on the surface of an ideal fluid.

After revisiting the derivation of the systems through asymptotic techniques, we conduct an asymptotic analysis of the mass, momentum, and energy balance laws. Theoretical justifications are provided for these balance laws, along with the approximations of the velocity field within the systems under consideration. These justifications are based on error estimates between solutions of the Euler equations and the Boussinesq systems, assuming identical initial data.

This paper is structured as follows: In Section \ref{sec:derivation}, we provide an asymptotic derivation of the two-dimensional Boussinesq systems found in \cite{BCL2015, BCS2002}. In Section \ref{sec:rigorous}, we offer a mathematically rigorous derivation of the balance laws, expanding significantly on the preliminary work presented in \cite{IK2020, IK_PLA19}. Additionally, we conduct an analysis of the errors in the balance laws. We introduce essential theorems to enhance the understanding of the relationship between the formal derivation of the balance laws and their theoretical justification. These mechanical balance laws are dependent on the conventional small parameters $\alpha$ and $\beta$, which quantify the nonlinearity and dispersion of the system, respectively. We validate our theoretical findings numerically in Section \ref{sec:numerics}. The numerically computed errors in mass, momentum, and energy conservation align with the theoretically derived estimates for the propagation of two-dimensional water waves. We employ linear analysis to elucidate the behavior of the solutions and their adherence to balance laws. Furthermore, our estimation of the solutions of the linearized Boussinesq equations reveals a distinct behavior of the two-dimensional solutions compared to what is known for one-dimensional problems, as documented in \cite{Whitham1974}.

\section{Notation}

We start by introducing some notation. In what follows we denote by $X\in\R^2$ the horizontal variables $X=(x,y)$. We denote by $\alpha$ and $\beta$ the non-linearity and shallowness and parameters respectively given by \eqref{alphabeta}. We use the following notations for the gradient and laplacian:
\begin{align*}
\nabla=(\dx,\dy)^T,\qquad \Delta=\dx^2+\dy^2,\qquad \nabla_{X,z}=(\dx,\dy,\dz)^T,\qquad \nabla_{X,z}^{\beta}=(\sqrt{\beta}\,\nabla,\dz)^T.
\end{align*}
We denote by $e_z$ the upward normal unit vector in the vertical direction, while $\partial_n u$  is the upward co-normal derivative of $u$. The non-dimensional fluid domain will be defined as
\begin{equation*}
\Omega_t=\left\{(X,z)\in \R^3, 0\leq z \leq 1+\alpha\eta(t,X)\right\}\ .
\end{equation*}
For all $a,b\in\R$ we write $a\vee b=\max\{a,b\}$. We denote by $C(\lambda_1, \lambda_2,\cdots)$ a constant depending on the parameters $\lambda_1$, $\lambda_2,\cdots$, and \emph{whose dependence on the $\lambda_j$ is always assumed to be non-decreasing}.

Let $p$ be a constant with $1\leq p< \infty$ and $L^p(\R^2)$ the space of all Lebesgue-measurable functions $f$ with the standard norm $$\vert f \vert_{L^p(\R^2)}=\big(\int_{\R^2}\vert f(X)\vert^p dX\big)^{1/p}<\infty.$$ 
The space $L^\infty(\R^2)$ consists of all essentially bounded, Lebesgue-measurable functions $f$ with the norm
$$
\vert f\vert_{L^\infty(\R^2)}= \esssup \vert f(X)\vert<\infty\ .
$$
Let $f:(-1,0)\mapsto H^s(\R^2)$ be a Bochner measurable function. The space  $L^\infty H^s=L^{\infty}((-1,0);H^{s}(\R^2))$ is endowed with
the canonical norm: 
$$
\vert f\vert_{L^\infty H^s}=\esssup_{z\in(-1,0)} \vert f(X,z)\vert_{H^s(\R^2)}\ .
$$
For any real constant $s$, $H^s=H^s(\R^2)$ denotes the Sobolev space of all tempered distributions $f$ with the norm $\vert f\vert_{H^s}=\vert \Lambda^s f\vert_{L^2} < \infty$, where $\Lambda$ is the pseudo-differential operator $\Lambda=(1-\Delta)^{1/2}$.
We denote by $(H^{s,k},\vert\cdot\vert_{H^{s,k}})$ the Banach space defined by
\begin{align*}
H^{s,k}=\displaystyle\bigcap_{j=0}^{k}H^j\left((-1,0);H^{s-j}(\R^2)\right),  \qquad \vert\mathit{u}\vert_{H^{s,k}}=\sum_{j=0}^k \vert\Lambda^{s-j}\dz^j \mathit{u}\vert_{L^2}
\end{align*}

We denote by $\dot{H}^{s+1}(\R^2)$ the topological vector space
\begin{align*}
\dot{H}^{s+1}(\R^2)=\left\{f\in L^2_{loc}(\R^2),\quad \nabla f \in H^s(\R^2)\times H^s(\R^2)\right\},
\end{align*}
equipped with the semi-norm $\vert f \vert_{\dot{H}^{s+1}(\R^2)}=\vert\nabla f\vert_{H^s(\R^2)}$.





\section{Formal derivation of the Boussinesq systems}\label{sec:derivation}

We start with the derivation of the $a-b-c-d$ family of Boussinesq systems of \cite{BCL2015}.
We consider the domain $\left \{ (x,y,z) \in \mathbb{R}^3, -h_0<z<\eta (x,y,t)\right \}$ where the parameter
$h_0$  represents the undisturbed depth of the fluid and $\eta(x,y,t)$ represents the free surface deviation above its rest position. If $\boldsymbol{u}(x,y,z,t)=(u(x,y,z,t),v(x,y,z,t))$ denotes the horizontal components and $w(x,y,z,t)$ the vertical component of the fluid velocity vector field,  the assuming an ideal and irrotational flow we can define the velocity potential $\Phi$ as 
\begin{equation} \label{eq:solve}
 \boldsymbol{u}=\nabla \Phi, \ \ \ \ \ \  \ \   w=\Phi_z \ ,
\end{equation}
where $\nabla$ is the two-dimensional gradient operator
$(\frac{\partial }{\partial x},\frac{\partial }{\partial y})$.
Then, the Bernoulli equation and the free-surface boundary 
condition governing the motion of the fluid are formulated in terms of the potential and the surface excursion by
\begin{equation} \label{eq:freesurface}
 \left.\begin{matrix}
\Phi_t+\textstyle\frac{1}{2}(\Phi_x^2+\Phi_y^2+\Phi_z^2)+g \eta=0 \\
\eta_t+\Phi_x \eta_x+\Phi_y \eta_y-\Phi_z=0
\end{matrix}\right\} \quad \text{ on } \ \  z=\eta(x,y,t)\ ,
\end{equation}
where $g$ is is the gravitational acceleration, \cite{Whitham1974}. The derivation of asymptotic models relies on appropriate scaling. Here, we use non-dimensional variables:
\begin{subequations} \label{eq:nondimensional}
\begin{eqnarray}
 \tilde{x}=\frac{x}{l}, \ \  \tilde{y}=\frac{y}{l}, \ \ \ \tilde{z}=\frac{z+h_0}{h_0},
 \ \ \ \tilde{t}=\frac{\sqrt{gh_0}t}{l}\ ,\\
\tilde{\eta}=\frac{\eta}{A}, \ \ \ \tilde{\Phi}=\frac{h_0}{Al\sqrt{gh_0}}\Phi\ ,
\end{eqnarray}
\end{subequations}
where tilde denotes non-dimensional variables, $l$ and $A$ denote a characteristic wavelength and wave amplitude.
The governing equations and boundary conditions for the fully dispersive and fully non-linear irrotational water
wave problem are given by the Euler equations:
\begin{subequations} \label{EulerEqs}
 \begin{eqnarray}
 \beta  \nabla^2 \tilde{\Phi} + \tilde{\Phi}_{zz} &=&0,  \ \ 0<\tilde{z}<\alpha \tilde{\eta }+1\ ,  \label{eq:wholesystem1} \\
\tilde{\Phi}_{\tilde{z}}&=&0, \ \ \tilde{z}=0\ , \label{eq:wholesystem2} \\
\tilde{\Phi}_{\tilde{t}}+\frac{\alpha }{2}\left ( \tilde{\Phi}_{\tilde{x} }^2+\tilde{\Phi}_{\tilde{y} }^2
+\frac{1}{\beta }\tilde{\Phi}_{\tilde{z} }^2\right )+
\tilde{\eta }&=&0,  \ \ \text{on } \tilde{z}=\alpha \tilde{\eta }+1\ ,\label{eq:wholesystem3}  \\ 
\tilde{\eta }_{\tilde{t}}+\alpha \left [ \tilde{\eta}_{\tilde{x}}  \tilde{\Phi}_{\tilde{x}} +\tilde{\eta}_{\tilde{y}} 
\tilde{\Phi}_{\tilde{y}}\right ]-\frac{1}{\beta }\tilde{\Phi}_{\tilde{z}}&=&0 ,  
\ \ \text{on } \tilde{z}=\alpha \tilde{\eta }+1\ ,  \label{eq:wholesystem4}
\end{eqnarray}
\end{subequations}
where $\alpha$ and $\beta$ are measures of nonlinearity and frequency dispersion defined as
\begin{equation} \label{alphabeta}
 \alpha=\frac{A}{h_0}, \ \ \ \ \ \  \ \   \beta =\frac{h_0^2}{l^2}\ .
\end{equation}

The derivation then follows by considering the following formal asymptotic series expansion of the velocity potential $\tilde{\Phi}$,
\begin{equation} \label{eq:asymptoticseries}
 \tilde{\Phi} (\tilde{x},\tilde{y},\tilde{z},\tilde{t})=\sum_{n=0}^{\infty } \tilde{z}^{n}\tilde{\Phi}^{n}(\tilde{x}\ .
 \tilde{y},\tilde{t})\ .
\end{equation}
Substituting (\ref{eq:asymptoticseries}) into  (\ref{eq:wholesystem1}) and the Neumann boundary condition at 
bottom (\ref{eq:wholesystem2}) we obtain a polynomial in $\tilde{z}$ and requiring the coefficient of each power of 
$\tilde{z}$ to vanish, we obtain the classical recurrence relation
\begin{equation} \label{eq:reccurance}
 \tilde{\Phi}^{n+2}=-\beta \frac{\nabla ^{2}\tilde{\Phi}^{n}}{(n+1)(n+2)}, \ \ n=1,2,3,\cdots\ .
\end{equation}
From (\ref{eq:solve}), (\ref{eq:asymptoticseries}) and  (\ref{eq:reccurance}) one obtains the expression for 
velocity potential in the form
\begin{equation} \label{eq:potential}
 \tilde{\Phi}(\tilde{x},\tilde{y},\tilde{z},\tilde{t})=\sum_{n=0}^{\infty } 
 (-1)^{n}\frac{\tilde{z}^{2n} }{(2n)!}\beta ^{n}\nabla ^{2n}\tilde{f}\ ,
\end{equation}
where $\tilde{f}=\tilde{\Phi} (\tilde{x},\tilde{y},0,\tilde{t})$
is the potential evaluated at the bed.
Substituting the above expression for $\tilde{\Phi}$
into  (\ref{eq:wholesystem3}) yields the relation
\begin{equation} \label{eq:sbfc}
 \tilde{f}_{\tilde{t}}-\frac{\beta}{2}  \Delta \tilde{f}_{\tilde{t}} 
 +\tilde{\eta }+\frac{\alpha }{2} |\nabla \tilde{f}|^2 =\mathcal{O}(\alpha \beta ,\beta ^2)\ .
\end{equation}
Differentiating (\ref{eq:sbfc}) with respect to $x$ and $y$ and expressing the equations in terms of the non-dimensional  velocities at the bottom
$ \tilde{f}_{\tilde{x}}= \tilde{u}$, $ \tilde{f}_{\tilde{y}}= \tilde{v}$, we have the following equations
\begin{subequations}
\begin{align}
& \tilde{u}_{\tilde{t}}+\tilde{\eta }_{\tilde{x}}-\frac{1}{2}\beta\Delta  \tilde{u}_{\tilde{t}}+\alpha \left ( \tilde{u}\tilde{u}_{\tilde{x}}+
 \tilde{v}\tilde{v}_{\tilde{x}} \right )=\mathcal{O}(\alpha \beta ,\beta ^2)\ ,  \label{eq:system1}\\
 & \tilde{v}_{\tilde{t}}+\tilde{\eta }_{\tilde{y}}-\frac{1}{2}\beta \Delta \tilde{v}_{\tilde{t}}+\alpha \left ( \tilde{u}\tilde{u}_{\tilde{y}}+
 \tilde{v}\tilde{v}_{\tilde{y}} \right )=\mathcal{O}(\alpha \beta ,\beta ^2)\ , \label{eq:system2}\\
& \tilde{\eta}_{\tilde{t}}+\tilde{u }_{\tilde{x}}+\tilde{v }_{\tilde{y}}-
\frac{1}{6}\beta \left (\Delta \tilde{u}_{\tilde{x}}+\Delta \tilde{v}_{\tilde{y}}\right )+
\alpha \left ( (\tilde{\eta}\tilde{u})_{\tilde{x}}+(\tilde{\eta}\tilde{v})_{\tilde{y}} \right )=
\mathcal{O}(\alpha \beta ,\beta ^2)\ . \label{eq:system3}
\end{align}
\end{subequations}
 Now we let $\tilde{U},\tilde{V}$ are the dimensionless velocities at a 
dimensionless height $\theta$ $ (0\leq \theta \leq 1)$ in the fluid column. A formal use of Taylor's formula shows that
\begin{subequations}
\begin{align}
& \tilde{\Phi}_{\tilde{x}}|_{\tilde{z}=\theta }= \tilde{U}=
\tilde{u}
-\frac{\theta^2}{2}\beta \Delta \tilde{u}+\frac{\theta^4}{24}\beta^2 \Delta^2
 \tilde{u}+\mathcal{O}(\beta ^3)\ , \label{eq:height1}\\
&  \tilde{\Phi}_{\tilde{y}}|_{\tilde{z}=\theta }= \tilde{V}=
\tilde{v}
-\frac{\theta^2}{2}\beta \Delta \tilde{v} +\frac{\theta^4}{24}\beta^2 \Delta^2  \tilde{v}+\mathcal{O}(\beta ^3)\ , \label{eq:height2}
\end{align}
\end{subequations}
 as $\beta \rightarrow 0$. Solving these equations for $\tilde{u}$ and $\tilde{v}$ we have that
 \begin{subequations} 
  \begin{align}
& \tilde{u}=
\left ( 1+ \frac{\theta^2}{2}\beta\Delta +\frac{5 \theta^4}{24}\beta^2 \Delta ^2 \right )
 \tilde{U}+\mathcal{O}(\beta ^3)\ , \label{eq:height21}\\
& \tilde{v}=
\left ( 1+ \frac{\theta^2}{2}\beta\Delta +\frac{5 \theta^4}{24} \beta^2 \Delta ^2 \right )
 \tilde{V}+\mathcal{O}(\beta ^3)\ . \label{eq:height22}
\end{align}
\end{subequations}
 We substitute the relations  (\ref{eq:height21}) and  (\ref{eq:height22}) into (\ref{eq:system1}),
(\ref{eq:system2}) and (\ref{eq:system3}) to obtain
\begin{subequations} 
\begin{align} 
& \tilde{U}_{\tilde{t}}+\tilde{\eta }_{\tilde{x}}+\frac{\beta}{2}(\theta ^2-1) \Delta \tilde{U}_{\tilde{t}}+\alpha \left ( \tilde{U}\tilde{U}_{\tilde{x}}+
 \tilde{V}\tilde{V}_{\tilde{x}} \right )=\mathcal{O}(\alpha \beta ,\beta ^2)\ ,  \label{eq:tsystem1}\\
&  \tilde{V}_{\tilde{t}}+\tilde{\eta }_{\tilde{y}}+\frac{\beta}{2}(\theta ^2-1) \Delta \tilde{V}_{\tilde{t}}+\alpha \left ( \tilde{U}\tilde{U}_{\tilde{y}}+
 \tilde{V}\tilde{V}_{\tilde{y}} \right )=\mathcal{O}(\alpha \beta ,\beta ^2)\ , \label{eq:tsystem2}\\
& \tilde{\eta}_{\tilde{t}}+\tilde{U }_{\tilde{x}}+\tilde{V }_{\tilde{y}}+
\frac{\beta}{2} \left [ \theta ^2-\frac{1}{3} \right ]\left (\Delta \tilde{U}_{\tilde{x}}+\Delta \tilde{V}_{\tilde{y}}\right )  
+\alpha \left ( (\tilde{\eta}\tilde{U})_{\tilde{x}}+(\tilde{\eta}\tilde{V})_{\tilde{y}} \right )=
\mathcal{O}(\alpha \beta ,\beta ^2)\ . \label{eq:tsystem3}
\end{align}
\end{subequations}
Observing that 
\begin{align}
& \tilde{U}_{\tilde{t}}+\tilde{\eta }_{\tilde{x}}=\mathcal{O}(\alpha, \beta)\ ,\\
&  \tilde{V}_{\tilde{t}}+\tilde{\eta }_{\tilde{y}}=\mathcal{O}(\alpha, \beta)\ ,\\
&  \tilde{\eta}_{\tilde{t}}+\tilde{U }_{\tilde{x}}+\tilde{V }_{\tilde{y}}=\mathcal{O}(\alpha, \beta) \ ,
\end{align}
we can write for all $\lambda, \mu\in\mathbb{R}$
\begin{align*}
\beta~ \Delta \left [  \tilde{U}_{\tilde{x}}+ \tilde{V}_{\tilde{y}} \right ] & = \lambda \beta ~ \Delta \left [  \tilde{U}_{\tilde{x}}+ \tilde{V}_{\tilde{y}} \right ] 
+(1-\lambda )\beta ~ \Delta \left [ \tilde{U}_{\tilde{x}}+ \tilde{V}_{\tilde{y}} \right ] \\
& = \lambda \beta ~ \Delta \left [  \tilde{U}_{\tilde{x}}+ \tilde{V}_{\tilde{y}} \right ] 
-(1-\lambda )\beta ~ \Delta \tilde{\eta}_{\tilde{t}} + \mathcal{O}(\alpha\beta, \beta^2) \ , \\
\end{align*}
and also
\begin{align*}
\beta~ \Delta  \tilde{U}_{\tilde{t}} & = \mu \beta ~ \Delta  \tilde{U}_{\tilde{t}}
+(1-\mu )\beta ~ \Delta  \tilde{U}_{\tilde{t}} \\
& = \mu \beta ~ \Delta  \tilde{U}_{\tilde{t}}
-(1-\mu )\beta ~ \Delta  \tilde{\eta }_{\tilde{x}} + \mathcal{O}(\alpha\beta, \beta^2) \ , \\
\end{align*}
and similarly,
$$
\beta~ \Delta  \tilde{V}_{\tilde{t}} =  \mu \beta ~ \Delta  \tilde{V}_{\tilde{t}}
-(1-\mu )\beta ~ \Delta  \tilde{\eta }_{\tilde{y}} + \mathcal{O}(\alpha\beta, \beta^2)\ .
$$

Substitution of these relations into (\ref{eq:tsystem1})--(\ref{eq:tsystem3}) leads to the following general $a-b-c-d$ Boussinesq system
\begin{subequations} 
\begin{align} 
 & \tilde{U}_{\tilde{t}}+\tilde{\eta }_{\tilde{x}}+\alpha \left ( \tilde{U}\tilde{U}_{\tilde{x}}+
 \tilde{V}\tilde{V}_{\tilde{x}} \right )
 +\beta a \Delta \tilde{\eta }_{\tilde{x}}-\beta b  \Delta \tilde{U}_{\tilde{t}}
=\mathcal{O}(\alpha \beta ,\beta ^2) \ , \label{eq:gsystem1}\\
 & \tilde{V}_{\tilde{t}}+\tilde{\eta }_{\tilde{y}}+\alpha \left ( \tilde{U}\tilde{U}_{\tilde{y}}+
 \tilde{V}\tilde{V}_{\tilde{y}} \right )
 +\beta a \Delta \tilde{\eta }_{\tilde{y}} -\beta b   \Delta  \tilde{V}_{\tilde{t}} 
=\mathcal{O}(\alpha \beta ,\beta ^2) \ , \label{eq:gsystem2}\\
& \tilde{\eta}_{\tilde{t}}+\tilde{U }_{\tilde{x}}+\tilde{V }_{\tilde{y}}+
\alpha \left ( (\tilde{\eta}\tilde{U})_{\tilde{x}}+(\tilde{\eta}\tilde{V})_{\tilde{y}} \right ) 
+\beta c \Delta \left (\tilde{U}_{\tilde{x}}+\tilde{V}_{\tilde{y}} \right )
- \beta d \Delta \tilde{\eta} _{\tilde{t}} =
\mathcal{O}(\alpha \beta ,\beta ^2)\ . \label{eq:gsystem3}
\end{align}
 \end{subequations}
 where
 \begin{equation}\label{eq:paramgs}
 \begin{aligned}
 & a= \frac{1}{2} (1-\theta ^2) \mu, \quad b= \frac{1}{2}(1-\theta ^2)(1-\mu),  \\
 & c=\frac{1}{2} \left ( \theta ^2-\frac{1}{3} \right )\lambda, \quad d= \frac{1}{2}\left ( \theta ^2-\frac{1}{3} \right )(1-\lambda)\ .
 \end{aligned}
 \end{equation}
After neglecting the high-order terms and writing the variables in dimensional form system (\ref{eq:gsystem1})--(\ref{eq:gsystem3}) is written as
  \begin{subequations} 
\begin{align} 
 & {U}_{{t}}+g{\eta }_{{x}}+ \left ( {U}{U}_{{x}}+
{V}{V}_{{x}} \right )
  +gh_0^2 a \Delta {\eta }_{{x}} -h_0^2 b  \Delta  {U}_{{t}} =0 \ ,  \label{eq:dsystem1} \\
& {V}_{{t}}+g{\eta }_{{y}}+ \left ( {U}{U}_{y}+
{V}{V}_{{y}} \right )
 +gh_0^2 a \Delta {\eta }_{{y}}
 -h_0^2b  \Delta  {V}_{{t}}
=0 , \label{eq:dsystem2} \\
& {\eta}_{{t}}+h_0\left ( {U }_{{x}}+{V }_{{y}} \right )+
\left ( ({\eta}{U})_{{x}}+({\eta}{V})_{{y}} \right )+h_0^3 c \Delta \left ( {U}_{x}+{V}_{{y}} \right )
- h_0^2d\Delta{\eta} _{{t}}
=
0\ . \label{eq:dsystem3}
\end{align}
\end{subequations}

\vskip 0.1in

\begin{remark}
One should note that by introducing another parameter $\nu$, 
the system \eqref{eq:gsystem1},\eqref{eq:gsystem2},\eqref{eq:gsystem3} can be generalized to the
$a-b-a_1-b_1-c-d$ system
  \begin{subequations} 
\begin{align*} 
 & \tilde{U}_{\tilde{t}}+\tilde{\eta }_{\tilde{x}}+\alpha \left ( \tilde{U}\tilde{U}_{\tilde{x}}+
 \tilde{V}\tilde{V}_{\tilde{x}} \right )
 +\beta a \Delta \tilde{\eta }_{\tilde{x}}-\beta b  \Delta \tilde{U}_{\tilde{t}}
=\mathcal{O}(\alpha \beta ,\beta ^2) \ , \\
 & \tilde{V}_{\tilde{t}}+\tilde{\eta }_{\tilde{y}}+\alpha \left ( \tilde{U}\tilde{U}_{\tilde{y}}+
 \tilde{V}\tilde{V}_{\tilde{y}} \right )
 +\beta a_1 \Delta \tilde{\eta }_{\tilde{y}} -\beta b_1   \Delta  \tilde{V}_{\tilde{t}} 
=\mathcal{O}(\alpha \beta ,\beta ^2) \ , \\
& \tilde{\eta}_{\tilde{t}}+\tilde{U }_{\tilde{x}}+\tilde{V }_{\tilde{y}}+
\alpha \left ( (\tilde{\eta}\tilde{U})_{\tilde{x}}+(\tilde{\eta}\tilde{V})_{\tilde{y}} \right ) 
+\beta c \Delta \left (\tilde{U}_{\tilde{x}}+\tilde{V}_{\tilde{y}} \right )
- \beta d \Delta \tilde{\eta} _{\tilde{t}} =
\mathcal{O}(\alpha \beta ,\beta ^2)\, , 
\end{align*}
 \end{subequations}
where
 \begin{equation*}
\begin{aligned}
 & a_1 = \frac{1}{2} (1-\theta ^2) \nu, \quad b_1 = \frac{1}{2}(1-\theta ^2)(1-\nu).
 \end{aligned}
\end{equation*}
Choosing distinct parameters $\mu$ and $\nu$ could be of interest when studying waves
in the nearshore, where the dominant wave direction is approximately normal to the shoreline.
However, in the present work, we will stick to the four-parameter system 
\eqref{eq:gsystem1},\eqref{eq:gsystem2},\eqref{eq:gsystem3} for the sake of tidiness.
\end{remark}

In order to compute the associated mass, momentum and energy densities and fluxes, we need expressions for the
 velocities and pressure. The velocity field can be easily computed using  (\ref{eq:height1})--(\ref{eq:height2}). The expression for the pressure is obtained from Bernoulli's equation,
 \begin{equation}
 \Phi_t+\frac{1}{2}\left | \nabla_{X,z} \Phi \right |^2=-\frac{P}{\rho }-gz+C\ .
\end{equation}
We can find the constant $C$ by evaluating the previous equation at the free surface $z=\eta$. Specifically, we find
\begin{equation}
  C=\frac{P_{atm}}{\rho }\ ,
\end{equation}
where $P_{atm}$ refers to the atmospheric pressure. We introduce the dynamic pressure with the equation
\begin{equation}
 P'=P-P_{atm}+\rho g z\ ,
\end{equation}
which can be scaled using a typical wave amplitude  by $\rho g A \tilde{P}'=P'$, then
\begin{align*} 
 \tilde{P}'&=-\tilde{\Phi}_{\tilde{t}}-\frac{1}{2}\alpha (\tilde{\Phi}_{\tilde{x}}^2+\tilde{\Phi}_{\tilde{y}}^2)
 -\frac{1}{2}\frac{\alpha }{\beta }(\tilde{\Phi}_{\tilde{z}}^2)\\
&=-\tilde{f}_{\tilde{t}}+ \beta \frac{\tilde{z}^2}{2} \Delta \tilde{f}_{\tilde{t}} 
-\frac{\alpha}{2} \left [ \tilde{f}_{\tilde{x}}^2+\tilde{f}_{\tilde{y}}^2 \right ]+\mathcal{O}(\alpha \beta ,\beta ^2)\ .
\end{align*}
If we use the relation (\ref{eq:sbfc}) and recall the relation 
$ \tilde{f}_{\tilde{x}\tilde{x}\tilde{t}} +\tilde{f}_{\tilde{y}\tilde{y}\tilde{t}}=\tilde{U}_{\tilde{x}\tilde{t}} 
+\tilde{V}_{\tilde{y}\tilde{t}}+\mathcal{O}( \beta )$, then the scaled dynamic pressure becomes
\begin{equation}\label{eq:dpresure}
 \tilde{P}'=\tilde{\eta }+\frac{1}{2}\beta (\tilde{z}^2-1)[\tilde{U}_{\tilde{x}\tilde{t}} 
+\tilde{V}_{\tilde{y}\tilde{t}}]+\mathcal{O}(\alpha \beta ,\beta ^2)\ .
\end{equation}
The  total pressure in-terms of dimensional variables is then given by
\begin{equation}
 P=P_{atm}-\rho g(z-\eta )+\frac{\rho }{2}\left [ (z+h_0)^2-h_0^2 \right ]({U}_{{x}{t}} 
+{V}_{{y}{t}})+\mathcal{O}(\alpha \beta ,\beta ^2)\ .
\end{equation}

The following section is devoted to a mathematically rigorous approach to understanding the validity of the
two-dimensional Boussinesq system as an approximation of the the water-wave problem represented by
the Euler equations \eqref{EulerEqs} with a particular focus on
justifiying approximate mass, momentum and energy balance laws similar to those
presented in \cite{AK2}, \cite{AK3}, \cite{KKM}.

\section{Rigorous approach to balance laws}\label{sec:rigorous}

As we showed in the previous section, appropriate assumptions on the respective magnitude of the parameters $\alpha$ and $\beta$, lead to the derivation of 
(simpler) asymptotic models from the Euler equations. The Stokes number
$$ S=\dsp\frac{\alpha}{\beta}\ ,$$ 
is introduced in order to quantify the applicability of the equation to a particular regime of surface water waves. For the Boussinesq regime, the Stokes number is usually considered to be of order $1$. Here, for the sake simplicity, we assume that the Stokes number is equal to 1 ($\alpha=\beta$), so that we can work with a single small parameter $\alpha$ or $\beta$. Throughout this section we will denote by $\scrO(\beta^n)$, with $n\in\mathbb{N}$ 
any family of functions $(f^\beta)_{\beta\in]0,1[}$
such that $\displaystyle\frac{1}{\beta^n}f^\beta$  remains bounded in $L^{\infty}([0,\frac{T}{\beta}],H^{r})$,
for all $\beta\ll1$ and for possibly different values of $r$.
It is noted that in the sequel we consider the Zakharov-Craig-Sulem formulation of the Euler equations.
Specifically, we consider the Euler equations in terms of the Dirichlet-Neumann operator as follows :
\begin{equation}\label{eq:zakharov}
\left\lbrace
\begin{aligned}
& \teta_{\tilde{t}} -\displaystyle \frac{1}{\beta} \mathcal{G}_{\beta}[\beta\teta] \psi = 0\ , \\
& \psi_{\tilde{t}} +  \teta + \displaystyle\frac{\beta}{2} \vert\nabla\psi\vert^2 
         - \displaystyle\frac{ [ \mathcal{G}_{\beta}[\beta\teta]\psi + \beta^2\nabla\teta\cdot \nabla\psi ]^2 }{2 (1+ \beta^3 \vert\nabla\teta\vert^2)} = 0\ .
\end{aligned}\right.
\end{equation} 
Given a solution of this system, we reconstruct the potential $\tilde{\Phi}$ by solving the Laplace
equation (\ref{eq:bvproblem}) below.  More precisely,  we introduce the trace of the velocity potential at the free surface, defined as
$$
\psi=\tilde{\Phi}_{\mid_{\tilde{z}=1+\beta\teta}}\ ,
$$
and the Dirichlet-Neumann operator  $\mathcal{G}_{\beta}[\beta\teta]\cdot$  as
$$
\mathcal{G}_{\beta}[\beta\teta]\psi=\partial_{\tilde{z}} \tilde{\Phi}_{\mid_{\tilde{z}=1+\beta\teta}}\ ,
$$
with $\tilde{\Phi}$ solving the boundary value problem
\begin{equation}\label{eq:bvproblem}
\left\{
\begin{array}{lcl}
\beta \partial_{\tilde{x}}^2\tilde{\Phi}+\beta\partial_{\tilde{y}}^2\tilde{\Phi}+ \partial_{\tilde{z}}^2 \tilde{\Phi}= 0\ ,\\
\partial_n\tilde{\Phi}_{\mid_{\tilde{z}=0}}=0\ ,\\
\tilde{\Phi}_{\mid_{\tilde{z}=1+\beta\teta}}=\psi\ .
\end{array}
\right.
\end{equation}
We look for an asymptotic expansion of $\tilde{\Phi}$ of the form
\begin{equation}\label{eq:phiapp}
\tilde{\Phi}^{app}=\sum_{j=0}^{N}\beta^j\tilde{\Phi}_j.
\end{equation}
Plugging this expression into the boundary value problem (\ref{eq:bvproblem}) one can cancel the residual up to the order $\mathcal{O}(\beta^{N+1})$ provided that
\begin{equation}\label{eq:phij}
\partial_{\tilde{z}}^2 \tilde{\Phi}_j = -\partial_{\tilde{x}}^2 \tilde{\Phi}_{j-1}-\partial_{\tilde{y}}^2\tilde{\Phi}_{j-1}, \quad j=0,\cdots,N\ ,
\end{equation}
(with the convention that $\tilde{\Phi}_{-1}=0$), together with the boundary conditions
\begin{equation}\label{eq:phij_bvp} 
\left\{
\begin{array}{ll}
\tilde{\Phi}_{j_{\mid_{\tilde{z}=1+\beta\teta}}}=\delta_{0,j}\psi\ ,\\
\partial_{\tilde{z}}\tilde{\Phi}_{j_{\mid_{\tilde{z}=0}}}=0\ ,
\end{array}
\right. \quad j=0,\cdots,N\ ,
\end{equation}
(where $\delta_{0,j}=1$ if $j=0$ and $0$ otherwise). Solving equation (\ref{eq:phij})
with the boundary conditions (\ref{eq:phij_bvp}) as in \cite{KZI2018}
one finds
$$\tilde{\Phi}_0=\psi\ ,$$
$$\tilde{\Phi}_1=-\frac{1}{2} \tilde{z}^2\Delta\psi+\frac{1}{2}\Delta\psi+\beta\teta\Delta\psi+\frac{1}{2}\beta^2\teta^2\Delta\psi \,$$
\begin{align*}
\tilde{\Phi}_2&=\frac{1}{24} \tilde{z}^4\Delta^2\psi-\frac{1}{4}\tilde{z}^2\Delta^2\psi+\frac{5}{24}\Delta^2\psi+\frac{5}{6}\beta\teta\Delta^2\psi-\frac{1}{2}\beta \tilde{z}^2\Delta\teta\Delta\psi+\frac{1}{2}\beta\Delta\teta\Delta\psi\\
&-\beta \tilde{z}^2\nabla\teta\cdot\nabla\Delta\psi+\beta\nabla\teta\cdot\nabla\Delta\psi-\frac{1}{2}\beta \tilde{z}^2\teta\Delta^2\psi+\scrO(\beta^2)\ .
\end{align*}



\subsection{Preliminary Results}\label{sec:main}


We first present results related to the formal approximations of the velocity potential. These results are necessary ingredients not only for the justification of the derivation of the Boussinesq systems but also of the justification of the mechanical balance laws that we will derive in the next section.

We denote by $\tPhi^{Asym}$, the formally defined potential in Section \ref{sec:derivation}, and we begin by decomposing $\tPhi^{Asym}$ into two parts; a finite part $P^{series}$ and a remainder term $R^{series}$,
$$
\tPhi^{Asym}:=\sum_{n=0}^{\infty}(-1)^n\frac{\tilde{z}^{2n}}{(2n)!}\beta^n\nabla^{2n} f =P^{series}+R^{series}\ , 
$$
where
\begin{equation}\label{eq:exp_P}
    P^{series}:=f-\beta \frac{\tilde{z}^2}{2}\Delta f+\beta^2\frac{\tilde{z}^4}{24}\Delta^2 f\ ,
\end{equation}
and
$$f:= \tilde{\Phi}\vert_{\tilde{z}=0}\ .$$
We can consider the following choice of $f$:
\begin{equation}\label{eq:explicit_f}
f=\psi+\beta\frac{1}{2}\Delta \psi+\beta^2\frac{5}{24}\Delta^2\psi+\beta^2\teta\Delta\psi\ .
\end{equation}
\begin{remark}
Note that our choice for $f$ turns out to be well justified (up to order $\beta^3$) using the following definition:
\begin{equation}
\tilde{\Phi}^{app}\vert_{\tilde{z}=0}=\psi+\beta\frac{1}{2}\Delta \psi+\beta^2\frac{5}{24}\Delta^2\psi+\beta^2\teta\Delta\psi+\scrO(\beta^3)\ .
\end{equation}
\end{remark}
Using the new notation $\tilde{u}=f_{\tilde{x}}$ and $\tilde{v}=f_{\tilde{y}}$, we easily verify the following expressions:
\begin{equation} \label{eq:height111}
P^{series}_{\tilde{x}}=
\tilde{u}
-\frac{\tilde{z}^2}{2}\beta \Delta\tilde{u}+\frac{\tilde{z}^4}{24}\beta^2 \Delta^2\tilde{u}\ ,
 \end{equation}
 and
 \begin{equation}\label{eq:height222}
  P^{series}_{\tilde{y}}=
\tilde{v}
-\frac{\tilde{z}^2}{2}\beta \Delta\tilde{v}+\frac{\tilde{z}^4}{24}\beta^2 \Delta^2\tilde{v} \ .
\end{equation}

\begin{theorem}\label{theorem:Pphi}
Let $(\eta^{Euler},\tilde{\Phi})$ be a regular solution  of the Euler system \eqref{EulerEqs}, such that $(\eta^{Euler},\nabla\psi)\in H^s(\R^2)\times H^{s}(\R^2)$ with $s$ large enough. Then, for $0<\tilde{t}<T/\beta$ we have,
$$\vert P^{series}-\tilde{\Phi}^{app}\vert_{L^{\infty}(\Omega_t)}\leq C \beta^3\ ,$$
where $C$ is uniform with respect to the parameter $\beta$.
\end{theorem}

\begin{proof}
We denote by $\teta=\eta^{Euler}$, and we make use of the following identity
\begin{align*}
P^{series}&=f-\beta \frac{\tilde{z}^2}{2}\Delta f+\beta^2\frac{\tilde{z}^4}{24}\Delta^2 f\\
&=\psi+\beta\frac{1}{2}\Delta \psi+\beta^2\frac{5}{24}\Delta^2\psi+\beta^2\teta\Delta\psi-\frac{1}{2}\beta \tilde{z}^2\Delta\psi-\frac{1}{4}\beta^2 \tilde{z}^2 \Delta^2\psi+\frac{1}{24}\beta^2 \tilde{z}^4\Delta^2\psi+\zeta\\
&=\tilde{\Phi}^{app}+h\ ,
\end{align*}
where $h=\beta^3\zeta$ and $\zeta$ is a function of $\tz$, $\teta$ and the derivatives of $\psi$ resulting from the asymptotic expansion of the potential given by $\tilde{\Phi}^{app}$. Looking for solutions
to the water-wave equations with $\psi$ in some Sobolev space would be too restrictive. Therefore, we can assume that the derivatives of $\psi$ belong to some Sobolev space $H^s(\R^2)$ with $s$ large enough so that $H^s(\R^2)\hookrightarrow L^{\infty}(\R^2)$. Using the fact that $\tilde{z}$ is bounded above by $1+\beta\teta$ and $(\eta^{Euler},\nabla\psi)\in H^s(\R^2)\times H^{s}(\R^2)$, we deduce that $h=\scrO(\beta^3)$, and hence the result follows.
\end{proof}
Since we aim in finding a rigorous justification of the formally derived mechanical balance laws for two-dimensional Boussinesq systems, we must first estimate approximations between the velocities $\nabla_{X,z} P^{series}$ and $\nabla_{X,z}\tPhi$ as well as for the pressure $\tP'$.

\subsection{Flattening the Domain}\label{sec:FlatDomain}

The aim of this subsection is to present the method of
"flattening the domain" of \cite{Lannes13} so as to normalize the domain of the differential equations.
Consider the boundary-value problem
\begin{align}\begin{cases}
\beta\partial_{\tilde{x}}^2 \omega + \beta\partial_{\tilde{y}}^2 \omega + \partial_{\tilde{z}}^2 \omega = R \textrm{ in } \Omega_t,\\
\omega\vert_{\tilde{z}=1+\beta\teta}=0,\\
\partial_n \omega\vert_{\tilde{z}=0}=0,
\end{cases}
\end{align}
where $\omega$ and $R$ are smooth functions. Our goal is to give estimates on $\omega$ and $\omega_t$ that justify the results obtained in theorems \ref{theorem:phi_euler_app} and \ref{theorem:tphi_euler_app}. To do that we transform the above system into an elliptic boundary value problem on a fixed domain. We assume that the water height is always bounded from below. In other words, 
\begin{align}\label{eq:h_min}
\exists h_{min} >0,\qquad \forall X=(x,y)\in \R^2,\qquad 1+\beta\teta\geq h_{min} .
\end{align}
\begin{figure}[h!]
\hspace{-7cm}
\begin{tikzpicture}[xscale=1,yscale=.75,font=\footnotesize]
\draw (0,0) -- (5,0);
\draw[thick,->] (0,3) -- (5,3) node[right] {$x$};
\node[draw] at (2.5, 1.5)   (a) {$\mathcal{S}$};
\draw[thick,->] (0,0) -- (0,5) node[right] {$z$};
\draw (-0.55,0) node[left] {$-1$};
\draw (-0.55,3) node[left] {$0$};
\coordinate (R) at (4,-0.5);
\coordinate (B) at (8,-0.5);
\node[] at (6,-1) {$\Sigma(x,z)$};
\draw[thin,->]    (R) to[out=-50,in=230] (B);
\hspace{7cm}
\draw[dashed] (0,3) -- (5,3);
\node[draw] at (2.5, 1.5)   (b) {$\Omega_t$};
\draw plot[domain=0:pi+2.15,samples=100,thick] (\x,{3+.5*cos(deg(\x))}) node[anchor=west]{$1+\alpha\eta(t,x)$};
\draw[thick,->] (0,0) -- (5,0) node[right] {$x$};
\draw[thick,->] (0,0) -- (0,5) node[right] {$z$};
\end{tikzpicture}
\caption{\small The transformation $\Sigma$ in a two-dimensional setting}%
\label{fig:transform}
\end{figure}
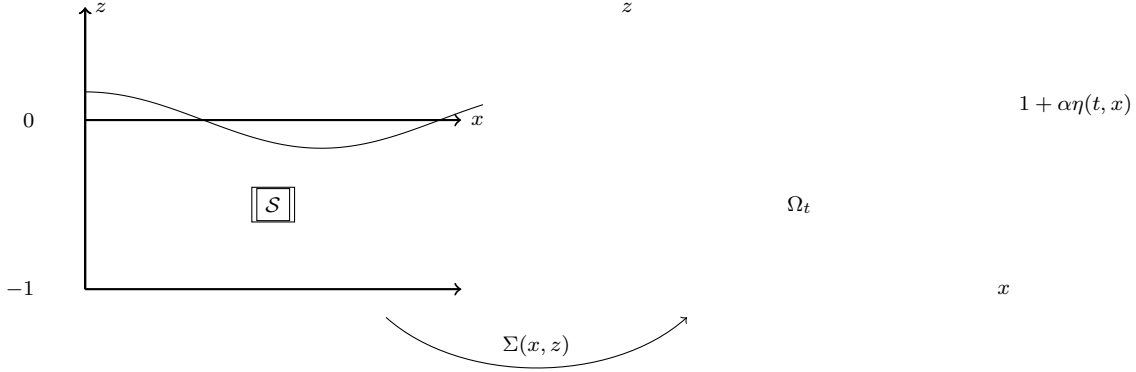%

We transform the variable domain $\Omega_t$ into a flat strip $S=\R^2\times (-1,0)$ by introducing the following diffeomorphism\footnote{We choose here the most obvious diffeomorphism but there are other choices of regularizing diffeomorphisms that are useful for obtaining optimal regularity estimates}  (see Figure \ref{fig:transform}):
\begin{align*}
\Sigma:\,\,&\quad S\longmapsto\Omega_t\\
 &(X,z)\rightarrow\Sigma(X,z)=(X,(1+\beta\teta)z+1+\beta\teta)
\end{align*} 
Then, the new variables $\mathbf{w}=\omega\circ\Sigma$ and $\mathbf{R}=R\circ\Sigma$ satisfy
the following elliptic boundary-value problem on the fixed domain $S$:
\begin{align}\begin{cases}\label{eq:flatdomain}
\displaystyle\frac{1}{1+\partial_{\tilde{z}}\sigma}\nabla_{X,z}^{\beta}\cdot P(\Sigma)\nabla_{X,z}^{\beta} \mathbf{w} = \mathbf{R} \textrm{ in } (-1,0)\times\R^2\\
\mathbf{w}\vert_{\tilde{z}=0}=0\\
e_z \cdot P(\Sigma)\nabla_{X,z}^{\beta} \mathbf{w}\vert_{\tilde{z}=-1}=0
\end{cases}
\end{align}
Here we use the notation 
\begin{align*} 
\sigma(X,\tilde{z})=\beta\teta\tilde{z}+1+\beta\teta,\qquad P(\Sigma)=I + Q(\Sigma),\qquad
 Q(\Sigma)=\left(\begin{array}{c c}
\partial_{\tilde{z}}\sigma I_{2} & -\sqrt{\beta}\nabla\sigma\\
-\sqrt{\beta}\nabla\sigma^T & \frac{-\partial_{\tilde{z}}\sigma+\beta\vert\nabla\sigma\vert^2}{1+\partial_{\tilde{z}}\sigma}
\end{array}\right). \qquad 
\end{align*}
\begin{theorem}\label{theorem:flat-domain}
Let  $\mathbf{R}\in H^{s,0},\, \mathbf{w}\in H^{s+1,1}, s\geq 0 $ solution of \eqref{eq:flatdomain} such that assumption \eqref{eq:h_min} is satisfied, then there exists a constant $C$ such that, for some $t_0>2$, the following estimate holds
\begin{align*}
\vert \Lambda^s \nabla_{X,z}^{\beta}\mathbf{w}\vert_{L^2}\leq C(h^{-1}_{min},\beta_{max},\vert\teta\vert_{H^{s+1\vee t_0+1}})\, \vert \Lambda^s \mathbf{R}\vert_{L^2}\, ,
\end{align*}
where $\beta_{max}$ is an upper bound of the shallowness parameter $\beta$. 
\end{theorem}
\begin{proof}
We briefly state the results, 
first for the case $s=0$. Multiplying the first equation of 
\eqref{eq:flatdomain} by $\mathbf{w}$ and integrating yields
\begin{align*}
\int_{S} \nabla_{X,z}^{\beta}\mathbf{w}\cdot P(\Sigma)\nabla_{X,z}^{\beta}\mathbf{w}=\int_S \mathbf{R}\mathbf{w}.
\end{align*}
Using the coercivity\footnote{The coercivity of $P(\Sigma)$ follows from
  the choice of the diffeomorphism, and we have  $\forall \theta\in\R^3, \,  \forall (X,z)\in \mathcal{S},\,  \exists K>0 / P(\Sigma)(X,z)\theta . \theta \geq K |\theta|^2$} of $P(\Sigma)$ along with  Poincar\'e inequality one gets the following estimate:
\begin{align*}
\vert\nabla_{X,z}^{\beta}\mathbf{w}\vert^2_{L^2}&\leq C(h_{min}^{-1},\beta_{\max},\vert\teta\vert_{H^{t_0+2}}) \vert\mathbf{R}\vert_{L^2}\vert\mathbf{w}\vert_{L^2},\\
&\leq C\,\vert\mathbf{R}\vert_{L^2}\vert\nabla_{X,z}^{\beta}\mathbf{w}\vert_{L^2},
\end{align*}
which proves the theorem for $s=0$. The result follows by an induction relation for $s>0$.
For details of this argument, the reader may consult \cite{Lannes13}.
\end{proof}
\begin{theorem}\label{theorem:phi_euler_app}
Let $(\eta^{Euler},\tilde{\Phi})$ be a regular solution  of the Euler system  such that $(\eta^{Euler},\nabla\psi)\in H^s(\R^2)\times H^{s}(\R^2)$ with $s$ large enough. Assume that the total water depth satisfies \eqref{eq:h_min}. Then, for  $0<\tilde{t}<T/\beta$ we have,
$$\vert\tilde{\Phi}_{\tilde{x}}^{app}-\tilde{\Phi}_{\tilde{x}}\vert_{L^{\infty}(\Omega_t)}\leq C\beta^2\ ,$$
$$\vert\tilde{\Phi}_{\tilde{y}}^{app}-\tilde{\Phi}_{\tilde{y}}\vert_{L^{\infty}(\Omega_t)}\leq C\beta^2\ ,$$
$$\vert\tilde{\Phi}_{\tilde{z}}^{app}-\tilde{\Phi}_{\tilde{z}}\vert_{L^{\infty}(\Omega_t)}\leq C\beta^2\ ,$$
where $C$ is a constant independent of $\beta$.
\end{theorem}
\begin{proof}
Let $\omega=\tilde{\Phi}^{app}-\tilde{\Phi}$. Since $\tilde{\Phi}$ is a solution of \eqref{eq:bvproblem} then its asymptotic expansion given by $\tilde{\Phi}^{app}$ satisfies
\begin{align*}\begin{cases}
\beta\partial_{\tilde{x}}^2\tilde{\Phi}^{app}+\beta\partial_{\tilde{y}}^2\tilde{\Phi}^{app}+\partial_{\tilde{z}}^2\tilde{\Phi}^{app}=\beta^3 r\,,\\
\tilde{\Phi}^{app}\vert_{\tilde{z}=1+\beta\teta}=\psi\,,\\
\partial_n \tilde{\Phi}^{app}\vert_{\tilde{z}=0}=0\,,\end{cases}
\end{align*}
where $r$ is a regular function in terms of $\tilde{z}$ and the derivatives of $\psi$.\newline 
Moreover, it is evident that
$$\partial_n \omega\vert_{\tilde{z}=0}=0 \quad \textrm{ and } \quad  \omega\vert_{\tilde{z}=1+\beta\teta}=0\ .$$
Then $\omega$ satisfies the following boundary-value problem:
\begin{equation}\label{eq:u_bvp}
\begin{cases}
\beta\partial_{\tilde{x}}^2 \omega + \beta\partial_{\tilde{y}}^2 \omega+\partial_{\tilde{z}}^2 \omega = \beta^3 r\ ,\\
\partial_n \omega\vert_{\tilde{z}=0}=0\ ,\\
\omega\vert_{\tilde{z}=1+\beta\teta}=0\ .
\end{cases}
\end{equation}
Let $\Sigma$ be the diffeomorphism defined in section \ref{sec:FlatDomain}. Using $\Sigma$, we transform system \eqref{eq:u_bvp} into a boundary value problem on a flat strip $S=(-1,0)\times\R^2$ and hence it follows by theorem \ref{theorem:flat-domain} that there exists a constant $C$ such that:
\begin{align*}
\vert \Lambda^s \nabla_{X,z}^{\beta}\mathbf{w}\vert_{L^2}\leq C \beta^3 \, \vert \Lambda^s \mathbf{r}\vert_{L^2}\,,
\end{align*}
where $\mathbf{w}=\omega\circ\Sigma$ and $\mathbf{r}=r\circ\Sigma$. We shall use the fact that $H^{s-1,1}(S)\hookrightarrow L^{\infty}((-1,0);H^{s-3/2}(\R^2))$ with the following estimate
\begin{align*}
\vert\nabla^{\beta}_{X,z}\mathbf{w}\vert_{H^{s-1,1}(S)}&=\vert \Lambda^{s-1}\nabla^{\beta}_{X,z}\mathbf{w}\vert_{L^2}+\vert\Lambda^{s-1}\nabla^{\beta}_{X,z}\dz\mathbf{w}\vert_{L^2}\\
 &\leq C(h_{min}^{-1},\beta_{max},\vert\teta\vert_{H^{s+1}},\vert Q\vert_{L^{\infty}H^s},\vert Q\vert_{H^{s,1}},\vert\dz Q\vert_{L^{\infty}H^s})\vert\Lambda^s \nabla_{X,z}^{\beta}\mathbf{w}\vert_{L^2}\ .
\end{align*}
Moreover, for $s$ large enough we have $H^{s-3/2}(\R^2)\hookrightarrow L^\infty(\R^2)$, therefore by using the above information we get
\begin{align*}
\vert\nabla^{\beta}_{X,z}\mathbf{w}\vert_{L^{\infty}(S)}&=\esssup_{(X,z)\in S}\vert\nabla^{\beta}_{X,z}\mathbf{w}\vert\\
 &\leq \vert\nabla^{\beta}_{X,z}\mathbf{w}\vert_{L^{\infty}((-1,0);H^{s-3/2}(\R^2))}\\
 &\leq C \beta^3. 
\end{align*}
where $C$ is a constant independent of $\beta$. 
The last estimate remains true on $\Omega_t$ for $\nabla^{\beta}_{X,z}\omega$. Since, after changing variables we have $$\nabla^{\beta}_{X,z}\omega=\nabla^{\beta}_{X,z}\left(\mathbf{w}\circ\Sigma^{-1}\right)=J_{\Sigma^{-1}}^T\left(\nabla^{\beta}_{X,z}\mathbf{w}\circ\Sigma^{-1}\right),$$
where the coefficients of the Jacobian matrix $J_{\Sigma^{-1}}^T$ are bounded.
\end{proof}
In the following corollary, we justify rigorously the equations (\ref{eq:height1}) and (\ref{eq:height2}),    
\begin{corollary}\label{corollary:pphi_deriv}
Let $(\eta^{Euler},\tilde{\Phi})$ be a regular solution  of the Euler system \eqref{EulerEqs} such that $(\eta^{Euler},\nabla\psi)\in H^s(\R^2)\times H^{s}(\R^2)$ with $s$ large enough. Then, for $0<\tilde{t}<T/\beta$ we have,
\begin{align}
\vert\tilde{\Phi}_{\tilde{x}}-P^{series}_{\tilde{x}}\vert_{L^{\infty}(\Omega_t)}\leq C\beta^2\ ,\\
\vert\tilde{\Phi}_{\tilde{y}}-P^{series}_{\tilde{y}}\vert_{L^{\infty}(\Omega_t)}\leq C\beta^2\ ,\\
\vert \tilde{\Phi}_{\tilde{z}}- P_{\tilde{z}}^{series}\vert_{L^{\infty}(\Omega_t)}\leq C\beta^2\ .
\end{align}
Where C is uniform with respect to the parameter $\beta$.
\end{corollary}
\begin{proof} 
We proceed by using a triangle inequality followed by previous estimates.
In fact, from Theorem \ref{theorem:Pphi} we have
\begin{align*}
P^{series}_{\tilde{x}}-\tPhi^{app}_{\tilde{x}}=h_{\tilde{x}},
\end{align*}
where $h_{\tilde{x}}=\beta^3\zeta_{\tilde{x}}$ and $\zeta$ is a function of $\tz$, $\teta$ and the derivatives of $\psi$ regular enough. Therefore, for $s$ large enough such that 
$H^s(\R^2)\hookrightarrow L^{\infty}(\R^2)$ with $\tz$ bounded above by $1+\beta\teta$, one can write
\begin{align*}
 \vert\tilde{\Phi}_{\tilde{x}}- P^{series}_{\tilde{x}}\vert_{L^{\infty}(\Omega_t)}&\leq \vert \tilde{\Phi}_{\tilde{x}}-\tilde{\Phi}_{\tilde{x}}^{app}\vert_{{L^{\infty}(\Omega_t)}}+\vert\tilde{\Phi}_{\tilde{x}}^{app}- P^{series}_{\tilde{x}}\vert_{{L^{\infty}(\Omega_t)}}\\
 &\leq C\beta^2\ ,
 \end{align*}
where C is uniform with respect to the parameter $\beta$. 
Similar considerations apply to the second and third inequalities.
\end{proof}

\subsection{Expression for the pressure}\label{sec:pressure}

The aim of this section is to find an approximation of the pressure defined in the context of 
the Euler equations. We begin with a useful remark derived from the Zakharov-Craig-Sulem equations, 
cf. \cite{La-Bo}. 

\begin{remark}\label{remark:zakharov}
Let $\teta\in H^{s+1/2}(\R^2)\cap H^{t_0+2}(\R^2)$ with $s\geq 0$, $t_0>1$ satisfying \eqref{eq:h_min}.
Then, the following mappings are continuous:
\begin{align}
\mathcal{G}_{\beta}[\beta\teta]\colon \dot{H}^{s+1}(\R^2)&\to H^{s-1/2}(\R^2)\\
\psi &\mapsto \mathcal{G}_{\beta}[\beta\teta]\psi
\end{align}
\begin{align}
\nu[\beta\teta]\colon 
\dot{H}^{s+1/2}(\R^2)&\to H^{s-1/2}(\R^2)\\
\psi &\mapsto \frac{ [ \mathcal{G}_{\beta}[\beta\teta]\psi + \beta^2\nabla\teta\cdot \nabla\psi ]^2 }{2 (1+ \beta^3 \vert\nabla\teta\vert^2)}
\end{align}
\end{remark} 

\begin{theorem}\label{theorem:tphi_euler_app}
Let $(\eta^{Euler},\tilde{\Phi})$ be a regular solution  
of the Euler equations, such that $(\eta^{Euler},\nabla\psi)\in H^s(\R^2)\times H^{s}(\R^2)$ 
with $s$ large enough. Assume that the total water depth satisfies \eqref{eq:h_min}. 
Then, for $0<\tilde{t}<T/\beta$ we have,
$$\vert\tilde{\Phi}^{app}_{\tilde{t}}-\tPhi_{\tilde{t}}\vert_{L^{\infty}(\Omega_t)}\leq C\beta^2\,,$$
where $C$ is independent of $\beta$.
\end{theorem}
\begin{proof}

Let $\left(\teta,\nabla\psi\right)\in  H^s(\R^2)\times H^{s}(\R^2)$ with $s$ large enough. As a result of Remark \ref{remark:zakharov} we have, 
$$\frac{ [ \mathcal{G}_{\beta}[\beta\teta]\psi + \beta^2\nabla\teta\cdot \nabla\psi ]^2 }{2 (1+ \beta^3 \vert\nabla\teta\vert^2)}\in H^{s-1/2}(\R^2)\ .$$
In fact, this quantity appears in the expression of $\psi_{\tilde{t}}$ as shown in system \eqref{eq:zakharov}:
$$\psi_{\tilde{t}} = -  \teta - \displaystyle\frac{\beta}{2} \vert\nabla\psi\vert^2 
         + \displaystyle\frac{ [ \mathcal{G}_{\beta}[\beta\teta]\psi + \beta^2\nabla\teta\cdot \nabla\psi ]^2 }{2 (1+ \beta^3 \vert\nabla\teta\vert^2)}\ .$$
Hence $\psi_{\tilde{t}}\in H^{s-1/2}(\R^2)$, and for $s$ large enough we have 
$H^{s-1/2}(\R^2)\hookrightarrow L^{\infty}(\R^2)$.\newline
We will use the notation $\omega=\tPhi^{app}-\tPhi$. Then $\omega$ satisfies the boundary value problem \eqref{eq:u_bvp}. In fact, if we denote $\phi=\tilde{\Phi}\circ\Sigma$ then $\phi$ satisfies:
\begin{align}\begin{cases}\label{eq:flat_phi}
\nabla_{X,z}^{\beta}\cdot P(\Sigma)\nabla_{X,z}^{\beta} \phi = 0 \textrm{ in } \R^2\times(-1,0)\\
\phi\vert_{\tilde{z}=0}=\psi\\
e_z \cdot P(\Sigma)\nabla_{X,z}^{\beta} \phi\vert_{\tilde{z}=-1}=0
\end{cases}
\end{align}
We look for an approximate solution to the above system of the form
\begin{align*}
\phi^{app}=\sum_{j=0}^N \beta^j\phi_j.
\end{align*}
By replacing $\phi^{app}$ in \eqref{eq:flat_phi} and canceling higher order terms in $\beta$ one obtains:
\begin{align}
\frac{1}{H}\partial^2_{\tilde{z}}\phi_0=0,\qquad \phi_0\vert_{\tilde{z}=0}=\psi,\qquad \frac{1}{H}\partial_{\tilde{z}}\phi_0=0,
\end{align}
and for all $1\leq j \leq n$,
\begin{align}\begin{cases}
\displaystyle\frac{1}{H}\partial^2_{\tilde{z}}\phi_j=-A(\nabla,\partial_{\tilde{z}})\phi_{j-1},\\
\displaystyle\phi_{j}\vert_{\tilde{z}=0}=\psi,\qquad \frac{1}{H}\partial_{\tilde{z}}\phi_j=0.
\end{cases}
\end{align}
The operator $A(\nabla,\dz)$ is given by
\begin{align*}
A(\nabla,\partial_{\tilde{z}})\bullet=\nabla\cdot(H\nabla\bullet)+\partial_{\tilde{z}}(\frac{\vert\nabla\sigma\vert^2}{H}\partial_{\tilde{z}}\bullet)-\nabla\cdot(\nabla\sigma\,\partial_{\tilde{z}}\bullet)-\partial_{\tilde{z}}(\nabla\sigma\cdot\nabla\bullet),
\end{align*}
where $H(t,X)=1+\beta\teta(t,X)$ denotes the water depth and $\sigma(t,X,\tilde{z})=\beta\teta(t,X) \tilde{z}+(1+\beta\teta(t,X))$.\newline
Then, $\phi_0=\psi$ and hence we can write:
\begin{align*}
\phi=\psi+h,
\end{align*}
where $h$ is function of $\tz$, $\teta$ and the derivatives of $\psi$. Thus, for $\mathbf{w}=\omega\circ\Sigma$ one has $\mathbf{w}_{\tilde{t}}\in H^{s-1/2}(\R^2)$. (One can choose initially $(\teta,\nabla\psi)\in H^{s+t_0}(\R^2)\times H^{s+t_0}(\R^2)$ with $t_0$ large enough.) Differentiating \eqref{eq:flatdomain} with respect to time, and for the sake of brevity (we refer to Lemma $5.4$ in \cite{Lannes13}) one readily obtains the following estimate: 
\begin{align*}
\vert\Lambda^s \nabla^{\beta}_{X,z} \mathbf{w_{\tilde{t}}}\vert_{L^2}\leq \beta^3 C(h^{-1}_{min},\beta_{max},\vert\teta\vert_{H^{s+t_0}},\vert\nabla\psi\vert_{H^{s+t_0}}).
\end{align*}
Using the embedding $H^{s-1,1}\hookrightarrow L^{\infty}((-1,0);H^{s-3/2})$ 
as in Theorem \ref{theorem:phi_euler_app}  one obtains the result.
\end{proof}

\begin{theorem}\label{theorem:t_pphi}
Let $(\eta^{Euler},\tilde{\Phi})$ be a regular solution  of the Euler system  such that $(\eta^{Euler},\nabla\psi)\in H^s(\R^2)\times H^{s}(\R^2)$ with $s$ large enough. Then, for $0<\tilde{t}<T/\beta$ we have,
$$\vert P^{series}_{\tilde{t}}-\tPhi_{\tilde{t}}\vert_{L^{\infty}(\Omega_t)}\leq C\beta^2\ ,$$
where $C$ is uniform with respect to the parameter $\beta$.
\end{theorem}
\begin{proof}

The following identity holds (we refer to the proof of Theorem \ref{theorem:Pphi} for clarification):
 $$P^{series}=\tilde{\Phi}^{app}+h \ ,$$
where $h=\mathcal{O}(\beta^3)$ is a function of $\tz$, $\teta$ and the derivatives of $\psi$.
Hence, we can write
$$P^{series}_{\tilde{t}}=\tilde{\Phi}^{app}_{\tilde{t}}+h_{\tilde{t}}\ ,$$
and we aim to show that $h_{\tilde{t}}\in H^{s-1/2}(\R^2)$ for $s$ large enough. In fact, from the second equation in system \eqref{eq:zakharov} we have
 \begin{align*}
 \psi_{\tilde{t}} =  -\teta - \frac{\beta}{2} \vert\nabla\psi\vert^2 +\frac{ [ \mathcal{G}_{\beta}[\beta\teta]\psi + \beta^2\nabla\teta\cdot \nabla\psi ]^2 }{2 (1+ \beta^3 \vert\nabla\teta\vert^2)}\ .  
 \end{align*}
 Then, as a consequence of Remark \ref{remark:zakharov}, $\psi_{\tilde{t}}\in H^{s-1/2}(\R^2)$.  Moreover, from the first equation in system \eqref{eq:zakharov} we have
 $$\teta_{\tilde{t}}=\frac{1}{\beta} \mathcal{G}_{\beta}[\beta\teta] \psi\,,$$
 which implies that $\teta_{\tilde{t}}\in H^{s-1/2}(\R^2)$.
Choosing $s$ large enough such that $H^{s-1/2}(\R^2)\hookrightarrow L^{\infty}(\R^2)$ and for $\tilde{z}$ bounded from above by $1+\beta\teta$, we can write:
 \begin{equation}
 \vert P_{\tilde{t}}^{series}-\tilde{\Phi}^{app}_{\tilde{t}}\vert_{L^{\infty}(\Omega_t)}\leq C\beta^2\ ,
 \end{equation}
 where $C$ is independent of $\beta$.
\end{proof}

The approximation of the pressure term $\tilde{P'}$ is given by the following corollary.
\begin{corollary}
Let $(\eta^{Euler},\tilde{\Phi})$ be a regular solution  of the Euler system  such that $(\eta^{Euler},\nabla\psi)\in H^s(\R^2)\times H^{s}(\R^2)$ with $s$ large enough. Then, for $0<\tilde{t}<T/\beta$ we have,
\begin{align}
\vert \tilde{P'}-\tilde{Q}\vert_{L^{\infty}(\Omega_t)}\leq C (\alpha\beta+\beta^2),
\end{align}
where $\tilde{Q}$ is an approximation of the pressure defined by:
\begin{align}
\tilde{Q}=-P^{series}_{\tilde{t}}-\frac{1}{2}\alpha \left((P^{series}_{\tilde{x}})^2+(P^{series}_{\tilde{y}})^2\right)
 -\frac{1}{2}\frac{\alpha}{\beta}(P^{series}_{\tilde{z}})^2\,,
\end{align}
and $C$ is a constant independent of $\beta$.
\end{corollary}
\begin{proof}
In fact, the definition of the pressure $\tilde{P}'$ is given in section \ref{sec:derivation} by
\begin{align}
\tilde{P}'=-\tPhi_{\ttt}-\frac{1}{2} \alpha\left( \tPhi_{\tx}^2+\tPhi_{\tilde{y}}^2\right)
 -\frac{1}{2}\frac{\alpha}{\beta}\tPhi_{\tz}^2\ ,
\end{align}
and thus one has
\begin{align}
\tilde{P}'-\tilde{Q}=&-(\tPhi_{\ttt}-P^{series}_{\tilde{t}})-\frac{1}{2} \alpha \left(\tPhi_{\tx}-P^{series}_{\tilde{x}}\right)\left(\tPhi_{\tx}+P^{series}_{\tilde{x}}\right)-\frac{1}{2} \alpha\left(\tPhi_{\tilde{y}}-P^{series}_{\tilde{y}}\right)\left(\tPhi_{\tilde{y}}+P^{series}_{\tilde{y}}\right)\\
 &-\frac{1}{2}\frac{\alpha}{\beta}\left(\tPhi_{\tz}-P^{series}_{\tilde{z}}\right)\left(\tPhi_{\tz}+P^{series}_{\tilde{z}}\right).
\end{align}
Given that $\tilde{\Phi}$ is regular enough and choosing $s$ sufficiently large such that $H^s(\R^2)\hookrightarrow L^{\infty}(\R^2)$ implies that $\nabla_{X,z}P^{series}\in L^{\infty}(\Omega_t)$ and consequently 
\begin{align*}
\vert\nabla_{X,z}\tPhi+\nabla_{X,z}P^{series}\vert_{L^{\infty}(\Omega_t)}&\leq \vert\nabla_{X,z}\tPhi\vert_{L^{\infty}(\Omega_t)}+\vert\nabla_{X,z}P^{series}\vert_{L^{\infty}(\Omega_t)}\\
&\leq k\ ,
\end{align*}
for some constant $k$. Hence using the previous estimates, one can deduce the result.
\end{proof}

\subsection{Mass balance}\label{sec:mass}

In this section, we establish the mass conservation properties of the general family of the two-dimensional 
Boussinesq systems (\ref{eq:gsystem1})--(\ref{eq:gsystem3}), (\ref{eq:paramgs}). The incompressibility of the fluid can be expressed in the form
$$ \rho _t+\frac{\partial }{\partial x}\left ( \rho u \right )+\frac{\partial }{\partial y}\left ( \rho v \right )
 +\frac{\partial }{\partial z}\left ( \rho w \right )=0\ ,$$
 which after integration and using Leibniz rule yields
$$
\begin{aligned}
\frac{\partial }{\partial t}\int_{-h_0}^{\eta }\rho \,dz-\rho \eta _t
&+\frac{\partial }{\partial x}\int_{-h_0}^{\eta }\rho u \,dz-\rho u |_{z=\eta } \eta _x \\
&+\frac{\partial }{\partial y}\int_{-h_0}^{\eta }\rho v \,dz-\rho v |_{z=\eta } \eta _y
+\int_{-h_0}^{\eta }\frac{\partial }{\partial z}(\rho w) \,dz=0\ .
\end{aligned}
$$
However, without loss of generality we assume that $\rho =1$ which simplifies the previous relation. Since the vertical velocity at the bottom $w |_{(-h_0)}=0$ and the kinematic boundary condition at the free surface of the water is $\eta_t+\Phi_x \eta_x+\Phi_y \eta_y-\Phi_z=0  \ \ \text{ on }\ \  z=\eta(x,y,t)\ ,$ 
we have
\begin{equation}
 \frac{\partial }{\partial t}\int_{-h_0}^{\eta } \,dz+\frac{\partial }{\partial x}\int_{-h_0}^{\eta } u \,dz 
 +\frac{\partial }{\partial y}\int_{-h_0}^{\eta } v \,dz=0\ ,
\end{equation}
or equivalently in non-dimensional variables
\begin{equation}\label{mass:phi}
 \frac{\partial }{\partial \tilde{t}}\left ( 1+\alpha \tilde{\eta } \right )
 +\frac{\partial }{\partial \tilde{x}}\int_{\tilde{z}=0}^{1+\alpha \tilde{\eta }}\alpha \tilde{\Phi}_{\tilde{x}}\,d\tilde{z}
 +\frac{\partial }{\partial \tilde{y}}\int_{\tilde{z}=0}^{1+\alpha \tilde{\eta }}\alpha\tilde{\Phi}_{\tilde{y}}  \,d\tilde{z}=0\ .
\end{equation}
Substituting the expressions for $\tilde{\Phi}_{\tilde{x}}$ and $\tilde{\Phi}_{\tilde{y}}$  in terms 
of $\tilde{U}$ and $\tilde{V}$ gives
\begin{equation} \label{eq:mb}
\begin{aligned}
 \frac{\partial }{\partial \tilde{t}}\left ({1+\alpha \tilde{\eta }} \right )
&+\frac{\partial }{\partial \tilde{x}}\left [ \tilde{U}(\alpha +\alpha ^2\tilde{\eta })
+\frac{\alpha \beta}{2} \left ( \theta ^2
 -\frac{1}{3} \right )\Delta \tilde{U} \right ]\\
&+\frac{\partial }{\partial \tilde{y}}\left [ \tilde{V}(\alpha +\alpha ^2\tilde{\eta })
+\frac{\alpha \beta}{2} \left ( \theta ^2
 -\frac{1}{3} \right )\Delta \tilde{V} \right ]
=\mathcal{O}(\alpha \beta^2 ,\alpha ^2 \beta )\ .
\end{aligned}
\end{equation}
Finally, we obtain the differential balance equation 
\begin{equation}
 \tilde{\eta }_{\tilde{t }}+\tilde{U}_{\tilde{x}}+\tilde{V}_{\tilde{y}}
+\alpha \left [  (\tilde{U}\tilde{\eta })_{\tilde{x}}+(\tilde{V}\tilde{\eta })_{\tilde{y}}\right ]
+\frac{\beta }{2}\left ( \theta ^2-\frac{1}{3} \right )(\Delta \tilde{U}_{\tilde{x}}
+\Delta \tilde{V}_{\tilde{y}})=\mathcal{O}(\alpha \beta,  \beta^2 )\ .
\end{equation}
From (\ref{eq:mb}) the non-dimensional mass density and the non-dimensional mass fluxes are 
\begin{align*}
 \tilde{M}&={1+\alpha \tilde{\eta }}\ ,\\
 \tilde{q}_{m_x}&= \tilde{U}(\alpha +\alpha ^2\tilde{\eta })+\frac{\alpha \beta}{2} \left ( \theta ^2
 -\frac{1}{3} \right )\Delta \tilde{U}\ ,\\
 \tilde{q}_{m_y}&= \tilde{V}(\alpha +\alpha ^2\tilde{\eta })+\frac{\alpha \beta}{2} \left ( \theta ^2
 -\frac{1}{3} \right )\Delta \tilde{V}\ .
\end{align*}
Thus, the mass balance is
\begin{equation}\label{eq:massD}
\frac{\partial }{\partial \tilde{t}}\tilde{M}+\frac{\partial }{\partial \tilde{x}}\tilde{q}_{m_x}
+\frac{\partial }{\partial \tilde{y}}\tilde{q}_{m_y}=\mathcal{O}(\alpha \beta^2,  \beta^3 )\ .
\end{equation}
 In dimensional variables the quantities in mass balance equation (\ref{eq:massD}) are the following:
 \begin{align*}
 {M}&={h_0+ {\eta }}\ ,\\
 {q}_{m_x}&= {U}(h_0+{\eta })+h_0^3 \frac{1}{2}\left ( \theta ^2-\frac{1}{3} \right )\Delta U\ ,\\
{q}_{m_y}&= {V}(h_0+{\eta })+h_0^3 \frac{1}{2}\left ( \theta ^2-\frac{1}{3} \right ) \Delta V \ .
\end{align*}

\subsection{Momentum balance}\label{sec:momentum}

In this section we find an approximate expression for momentum density and flux. for obtaining momentum balance,
we first consider the Euler equation (\ref{eq:ee}) and the incompressibility condition (\ref{eq:if})
\begin{align}
 \vec{u}_t+(\vec{u}\cdot\nabla )\vec{u}+\nabla P& =\vec{g}\ , \label{eq:ee}\\
\nabla \cdot\vec{u}& =0\ , \label{eq:if}
\end{align}
 where $\vec{u}=(u,v,w)$ represents the velocity field, $P(x,y,z,t)$ represents the pressure and $\vec{g}=(0,0,-g)$ represents  gravitational  vector. Writing the Euler equations in terms of the velocity potential $\vec{u}=\nabla\Phi$ we obtain the equations
\begin{align*}
& \Phi _{xt}+{(\Phi _{x}^2)}_x+{(\Phi _{x}\Phi _{y})}_y+{(\Phi _{x}\Phi _{z})}_z+P_x=0\ ,\\
& \Phi _{yt}+{(\Phi _{y}^2)}_y+{(\Phi _{y}\Phi _{x})}_x+{(\Phi _{y}\Phi _{z})}_z+P_y=0\ .
\end{align*}
Integrating over a fluid column and using the kinematic boundary condition (\ref{eq:freesurface}) yields
\begin{eqnarray*}
\frac{\partial }{\partial {t}}\int_{-h_0}^{\eta }\Phi _x \,dz
 +\frac{\partial }{\partial {x}}\int_{-h_0}^{{\eta }}\left ( \Phi _x^2+P \right )\,dz
 +\frac{\partial }{\partial {y}}\int_{-h_0}^{{\eta }}\Phi _x \Phi _y\,dz&=&0\ ,\\
\frac{\partial }{\partial {t}}\int_{-h_0}^{\eta }\Phi _y \,dz
 +\frac{\partial }{\partial {y}}\int_{-h_0}^{{\eta }}\left ( \Phi _y^2+P \right )\,dz
 +\frac{\partial }{\partial {x}}\int_{-h_0}^{{\eta }}\Phi _y \Phi _x\,dz&=&0\ .
\end{eqnarray*}
Expressing the above relations in non-dimensional variables leads to the equations
\begin{eqnarray*}
\alpha \frac{\partial }{\partial \tilde{t}}\int_{\tilde{z}=0}^{1+\alpha \tilde{\eta }}\tilde{\Phi}_{\tilde{x}}\,d\tilde{z}
+\frac{\partial }{\partial \tilde{x}}\int_{\tilde{z}=0}^{1+\alpha \tilde{\eta }} \left \{ \alpha ^2
(\tilde{\Phi}_{\tilde{x}}^2) 
+\alpha \tilde{P}'-(\tilde{z}-1)\right \}\,d\tilde{z}+
\alpha ^2 \frac{\partial }{\partial \tilde{y}} \int_{\tilde{z}=0}^{1+\alpha \tilde{\eta }}\tilde{\Phi}_{\tilde{x}}\tilde{\Phi}_{\tilde{y}}\,d\tilde{z}&=&0\ ,\\
\alpha \frac{\partial }{\partial \tilde{t}}\int_{\tilde{z}=0}^{1+\alpha \tilde{\eta }}\tilde{\Phi}_{\tilde{y}}\,d\tilde{z}
 +\frac{\partial }{\partial \tilde{y}}\int_{\tilde{z}=0}^{1+\alpha \tilde{\eta }} \left \{ \alpha ^2(\tilde{\Phi}_{\tilde{y}}^2
) +\alpha \tilde{P}'-(\tilde{z}-1)\right \}\,d\tilde{z}
+\alpha ^2 \frac{\partial }{\partial \tilde{x}} \int_{\tilde{z}=0}^{1+\alpha \tilde{\eta }}
\tilde{\Phi}_{\tilde{x}}\tilde{\Phi}_{\tilde{y}}\,d\tilde{z}&=&0\ .
\end{eqnarray*}
Substituting non-dimensional velocity potentials  $\tilde{\Phi}_{\tilde{x}}$ and $\tilde{\Phi}_{\tilde{y}}$ 
and $\tilde{P}'$ in terms of $\tilde{U}$ and $\tilde{V}$ gives the momentum balance equations
\begin{equation}
\begin{aligned}
\frac{\partial }{\partial\tilde{t } }\left \{   (1+\alpha \tilde{\eta})  \tilde{U}
+\frac{\beta }{2}\left ( \theta ^2-\frac{1}{3}\right ) \Delta \tilde{U} \right \}+
\frac{\partial }{\partial\tilde{x } }\left \{ \tilde{\eta}+\alpha \tilde{U}^2 
+\frac{\alpha }{2} \tilde{\eta }^2-\frac{1}{3}\beta (\tilde{U}_{\tilde{x }\tilde{t }}
 +\tilde{V}_{\tilde{y }\tilde{t }})+\frac{1}{2}\right \} \\
 +\frac{\partial }{\partial\tilde{y } }(\alpha \tilde{U}\tilde{V})=\mathcal{O}(\alpha \beta,  \beta^2 )\ , \label{eq:momentum1D}
 \end{aligned}
\end{equation}
 \begin{equation}\label{eq:momentum2D}
 \begin{aligned}
\frac{\partial }{\partial\tilde{t } }\left \{   (1+\alpha \tilde{\eta})  \tilde{V}
+\frac{\beta }{2}\left ( \theta ^2-\frac{1}{3}\right ) \Delta \tilde{V} \right \}+
\frac{\partial }{\partial\tilde{y } }\left \{ \tilde{\eta}+\alpha \tilde{V}^2 
+\frac{\alpha }{2} \tilde{\eta }^2-\frac{1}{3}\beta (\tilde{U}_{\tilde{x }\tilde{t }}
 +\tilde{V}_{\tilde{y }\tilde{t }})+\frac{1}{2}\right \}  \\
 +\frac{\partial }{\partial\tilde{x } }(\alpha \tilde{U}\tilde{V})=\mathcal{O}(\alpha \beta,  \beta^2 )\ . 
\end{aligned}
\end{equation}
 If the terms of order $\mathcal{O}(\alpha \beta,  \beta^2)$  are neglected, then the momentum balance equations written in dimensional variables take the following form
$$
\begin{aligned}
 & \frac{\partial }{\partial {t } }\left \{   (h_0+{\eta})  {U}
+\frac{1 }{2}\left ( \theta ^2-\frac{1}{3} \right )\Delta U\right \}+\frac{\partial }{\partial{x } }\left \{h_0{U}^2 
+\frac{g}{2} \left ( {h_0+\eta }\right )^2 -\frac{h_0^3}{3} ({U}_{{x }{t }}
 +{V}_{{y }{t }})\right \}
 +\frac{\partial }{\partial{y } }(h_0{U}{V})=0\ ,\\
 & \frac{\partial }{\partial {t } }\left \{   (h_0+{\eta})  {V}
+\frac{1 }{2}\left ( \theta ^2-\frac{1}{3} \right )\Delta V \right \}+\frac{\partial }{\partial{y} }\left \{h_0{V}^2 
+\frac{g}{2} \left ( {h_0+\eta }\right )^2 -\frac{h_0^3}{3} ({U}_{{x }{t }}
 +{V}_{{y }{t }})\right \}
 +\frac{\partial }{\partial{x } }(h_0{U}{V})=0\ .\\ 
\end{aligned}
$$

\subsection{Energy balance}\label{sec:energy}

The exact energy balance equation reads 
\begin{equation}
 \frac{\partial }{\partial t }\left \{ \frac{1}{2} \left | \nabla \Phi ^2 \right |+gz\right \}+\nabla \cdot \left \{ \left (  \frac{1}{2} \left | \nabla \Phi ^2 \right |+gz+P\right )\nabla \Phi  \right \}=0\ .
\end{equation}
Integrating over the a water column yields
$$
\begin{aligned}
 \frac{\partial }{\partial t} \left \{ \int_{-h_0}^{\eta }\frac{1}{2}\left | \nabla \Phi  \right | ^2 \,dz +\int_{0}^{\eta }gz \,dz\right \}&+\frac{\partial }{\partial {x}} \left \{\int_{-h_0}^{\eta }
\left (\frac{1}{2}\left | \nabla \Phi  \right | ^2+gz+P  \right ) \Phi _x \right \} \,dz\\
 & +\frac{\partial }{\partial {y}} \left \{\int_{-h_0}^{\eta }
 \left (\frac{1}{2}\left | \nabla \Phi  \right | ^2+gz+P  \right ) \Phi _y \right \} \,dz =0\ .
\end{aligned}
$$
Using non-dimensional variables the last equation reads
$$
\begin{aligned}
\frac{\partial }{\partial \tilde{t}} & \left \{ \int_{\tilde{z}=0}^{1+\alpha \tilde{\eta }}\frac{\alpha ^2}{2}
\left ( \tilde{\Phi}_{\tilde{x}}^2+\tilde{\Phi}_{\tilde{y}}^2+\frac{1}{\beta } \tilde{\Phi}_{\tilde{z}}^2\right )\,d\tilde{z}
+\int_{\tilde{z}=1}^{1+\alpha \tilde{\eta }}\left ( \tilde{z}-1 \right ) \,d\tilde{z}\right \}\\ 
&+\frac{\partial }{\partial \tilde{x}}\int_{\tilde{z}=0}^{1+\alpha \tilde{\eta }} 
\left \{ \frac{\alpha ^3}{2}\left ( \tilde{\Phi}_{\tilde{x}}^2+\tilde{\Phi}_{\tilde{y}}^2
+\frac{1}{\beta } \tilde{\Phi}_{\tilde{z}}^2 \right )
+\alpha \left ( \tilde{z}-1 \right )+\alpha ^2 \tilde{P}'
+\alpha \left (1- \tilde{z} \right )\right \}\tilde{\Phi}_{\tilde{x}}\,d\tilde{z}\\
&+\frac{\partial }{\partial \tilde{y}}\int_{\tilde{z}=0}^{1+\alpha \tilde{\eta }} 
\left \{ \frac{\alpha ^3}{2}\left ( \tilde{\Phi}_{\tilde{x}}^2+\tilde{\Phi}_{\tilde{y}}^2
+\frac{1}{\beta } \tilde{\Phi}_{\tilde{z}}^2 \right )+\alpha \left ( \tilde{z}-1 \right )+\alpha ^2 \tilde{P}'
+\alpha \left (1- \tilde{z} \right )\right \}\tilde{\Phi}_{\tilde{y}}\,d\tilde{z}=0\ .
\end{aligned}
$$
Substituting the expressions for $\tilde{\Phi}_{\tilde{x}}$ and $\tilde{\Phi}_{\tilde{y}}$  in terms of $\tilde{U}$ and $\tilde{V}$ and (\ref{eq:dpresure}) leads to the energy balance equation
\begin{align*}
\frac{\partial }{\partial \tilde{t}}& \left [ \frac{1}{2}(\tilde{U}^2
+\tilde{V}^2+\tilde{\eta}^2) +\frac{\beta }{2}
\left ( \theta ^2-\frac{1}{3} \right )(\tilde{U }\Delta\tilde{U }+\tilde{V}\Delta\tilde{V})
+\frac{\beta }{6}(\tilde{U }_{\tilde{x }}+\tilde{V}_{\tilde{y}})^2
+\frac{\alpha }{2}\tilde{\eta }(\tilde{U }^2+{\tilde{V}}^2)\right ]\\
&+\frac{\partial }{\partial \tilde{x}}\left [\frac{\alpha }{2} (\tilde{U}^3+\tilde{V}^2\tilde{U})
+\alpha  \tilde{\eta}^2 \tilde{U}+ \tilde{\eta} \tilde{U}+
\frac{\beta }{2}\left ( \theta ^2-\frac{1}{3}\right ) \tilde{\eta} \Delta \tilde{U}-
\frac{\beta }{3}\tilde{U}(\tilde{U}_{\tilde{x}\tilde{t}}+\tilde{V}_{\tilde{y}\tilde{t}})\right ]\\
&+\frac{\partial }{\partial \tilde{y}}\left [\frac{\alpha }{2} (\tilde{V}^3+\tilde{U}^2\tilde{V})
+\alpha  \tilde{\eta}^2 \tilde{V}+ \tilde{\eta} \tilde{V}+
\frac{\beta }{2}\left ( \theta ^2-\frac{1}{3}\right ) \tilde{\eta} \Delta \tilde{V}-
\frac{\beta }{3}\tilde{V}(\tilde{U}_{\tilde{x}\tilde{t}}+\tilde{V}_{\tilde{y}\tilde{t}})\right ]
=\mathcal{O}(\alpha \beta,  \beta^2 )\ .
\end{align*}
Hence, the general form of the energy balance equation is
\begin{eqnarray} \label{eq:energyD}
 \frac{\partial }{\partial \tilde{t}}\tilde{E}+\frac{\partial }{\partial \tilde{x}}\tilde{q}_{E_x}
+\frac{\partial }{\partial \tilde{y}}\tilde{q}_{E_y}=\mathcal{O}(\alpha \beta,  \beta^2 )\ . 
\end{eqnarray}
The dimensional form of the quantities are
\begin{eqnarray} \label{eq:en}
 E=\frac{1}{2} h_0\left (  U^2+V^2\right ) 
 +\frac{1}{2} h_0^3\left ( \theta ^2-\frac{1}{3} \right )\left ( U\Delta U+V\Delta V \right )\\ \nonumber
 +\frac{1}{2} \left(  U^2+V^2\right )\eta 
 +\frac{1 }{6} h_0^3 \left ( U_x+V_y \right )^2+\frac{1 }{2}g\eta ^2\ ,
\end{eqnarray}
and
\begin{eqnarray} \label{eq:enf}
q_{E_x}= \frac{1}{2} h_0 \left ( U^3+UV^2 \right ) +gh_0{\eta U} + \frac{1 }{2}gh_0^3\left ( \theta ^2-\frac{1}{3} \right )
\eta  \Delta U-\frac{1}{3} h_0^3U\left ( U_{xt}+V_{yt} \right )  +gU\eta ^2\ ,  
\end{eqnarray}
\begin{eqnarray} \label{eq:enf1}
q_{E_y}= \frac{1}{2} h_0 \left ( V^3+U^2 V\right ) +gh_0{\eta V} + \frac{1 }{2}gh_0^3\left ( \theta ^2-\frac{1}{3} \right )
\eta  \Delta V-\frac{1}{3} h_0^3V\left ( U_{xt}+V_{yt} \right )  +gV\eta ^2\ .  
\end{eqnarray}

\subsection{On the role of the rigorous approach }

Here we justify the terms $\mathcal{O}(\alpha \beta,  \beta^2 ) $ and  $\mathcal{O}(\alpha \beta^2,  \beta^3 )$  in the mechanical balance laws by justifying the formal use of Taylor's formula in \eqref{eq:height1} and \eqref{eq:height2}. We first verify that the expressions \eqref{eq:tsystem1}, \eqref{eq:tsystem2}, and \eqref{eq:tsystem3} hold for an explicit  definition of the variable $f$ given by \eqref{eq:explicit_f}.\newline
We focus on finding a common ground between the expressions used in Section \ref{sec:derivation} with new similar expressions satisfied by $P^{series}_{\tilde{x}}$ and $P^{series}_{\tilde{y}}$.\newline
In fact, from equations \eqref{eq:height111} and \eqref{eq:height222} one can write up to order $\mathcal{O}(\beta^3)$ :
\begin{align}
& P^{series}_{\tilde{x}}\vert_{\tilde{z}=\theta}=\tilde{U}=
\tilde{u}
-\frac{\theta^2}{2}\beta \Delta  \tilde{u}+\frac{\theta^4}{24}\beta^2 \Delta^2
 \tilde{u}\ , \label{eq:height333}\\
  & P^{series}_{\tilde{y}}\vert_{\tilde{z}=\theta}=\tilde{V}=
\tilde{v}
-\frac{\theta^2}{2}\beta \Delta \tilde{v}+\frac{\theta^4}{24}\beta^2 \Delta^2  \tilde{v}\ . \label{eq:height444}
\end{align}
One can verify that $f$ satisfies \eqref{eq:sbfc} and thus \eqref{eq:tsystem1}, \eqref{eq:tsystem2}, and \eqref{eq:tsystem3} hold in terms of the variables $\tilde{U}$ and $\tilde{V}$ which are now well justified in terms of $f$ and its derivatives. Therefore, giving a rigorous justification to the formally applied Taylor's formula used in \eqref{eq:height1} and \eqref{eq:height2}. This would not have been possible without using the fact that the approximations in Section \ref{sec:main} hold, specifically Corollary \ref{corollary:pphi_deriv} and Theorem \ref{theorem:t_pphi} that bring together the results and ensure that the use of $\nabla P^{series}$, $P_z^{series}$ and $P_t^{series}$ to approximate $\nabla\tilde{\Phi}$, $\tilde{\Phi}_z$ and $\tilde{\Phi}_t$ 
is valid up to order $\mathcal{O}(\beta^2)$.

\begin{theorem}
Let $(\eta^{Euler},\tilde{\Phi})$ be a regular solution  of the Euler system  such that $(\eta^{Euler},\nabla\psi)\in H^s(\R^2)\times H^{s}(\R^2)$ with $s$ large enough. Then, there exists constants $C_1,\,C_2,\,C_3,\,C_4$ independent of $\beta$ such that the following approximate mechanical balance laws are satisfied:
\begin{align*}
&\Big| \frac{\partial }{\partial \tilde{t}}(1+\alpha \tilde{\eta })+\frac{\partial }{\partial \tilde{x}}\left \{  \tilde{U}(\alpha +\alpha ^2\tilde{\eta })+\frac{\alpha \beta}{2} \left ( \theta ^2
 -\frac{1}{3} \right )\Delta \tilde{U}\right \}  \\\
&  +\frac{\partial }{\partial \tilde{y}}\left \{  \tilde{V}(\alpha +\alpha ^2\tilde{\eta })+\frac{\alpha \beta}{2} \left ( \theta ^2
 -\frac{1}{3} \right )\Delta \tilde{V}\right \}  \Big| _{L^{\infty}(\R^2)}\leq C_1 (\alpha \beta^2+ \alpha^2\beta)\ ,\\
&\Big|
\frac{\partial }{\partial\tilde{t } }\left \{   (1+\alpha \tilde{\eta})  \tilde{U}
+\frac{\beta }{2}\left ( \theta ^2-\frac{1}{3}\right ) \Delta \tilde{U} \right \}+
\frac{\partial }{\partial\tilde{x } }\left \{ \tilde{\eta}+\alpha \tilde{U}^2 
+\frac{\alpha }{2} \tilde{\eta }^2-\frac{1}{3}\beta (\tilde{U}_{\tilde{x }\tilde{t }}
 +\tilde{V}_{\tilde{y }\tilde{t }})+\frac{1}{2}\right \} \\
 &+\frac{\partial }{\partial\tilde{y } }(\alpha \tilde{U}\tilde{V})\Big| _{L^{\infty}(\R^2)}\leq C_2(\alpha \beta+  \beta^2 )\ ,\\ 
&\Big|
\frac{\partial }{\partial\tilde{t } }\left \{   (1+\alpha \tilde{\eta})  \tilde{V}
+\frac{\beta }{2}\left ( \theta ^2-\frac{1}{3}\right ) \Delta \tilde{V} \right \}+
\frac{\partial }{\partial\tilde{y } }\left \{ \tilde{\eta}+\alpha \tilde{V}^2 
+\frac{\alpha }{2} \tilde{\eta }^2-\frac{1}{3}\beta (\tilde{U}_{\tilde{x }\tilde{t }}
 +\tilde{V}_{\tilde{y }\tilde{t }})+\frac{1}{2}\right \}  \\
& +\frac{\partial }{\partial\tilde{x } }(\alpha \tilde{U}\tilde{V})\Big| _{L^{\infty}(\R^2)}\leq C_3(\alpha \beta+  \beta^2 )\ ,\\
&\Big| \frac{\partial }{\partial \tilde{t}}\tilde{E}+\frac{\partial }{\partial \tilde{x}}\tilde{q}_{E_x}
+\frac{\partial }{\partial \tilde{y}}\tilde{q}_{E_y}\Big| _{L^{\infty}(\R^2)}\leq C_4 (\alpha \beta+  \beta^2 )\ ,
\end{align*}
where $\tilde{E}$, $\tilde{q}_{E_x}$, and $\tilde{q}_{E_y}$ are given by:
\begin{align*}
&\tilde{E}= \frac{1}{2}(\tilde{U}^2
+\tilde{V}^2+\tilde{\eta}^2) +\frac{\beta }{2}
\left ( \theta ^2-\frac{1}{3} \right )(\tilde{U }\Delta\tilde{U }+\tilde{V}\Delta\tilde{V})
+\frac{\beta }{6}(\tilde{U }_{\tilde{x }}+\tilde{V}_{\tilde{y}})^2
+\frac{\alpha }{2}\tilde{\eta }(\tilde{U }^2+{\tilde{V}}^2),\\
&\tilde{q}_{E_x}=\frac{\alpha }{2} (\tilde{U}^3+\tilde{V}^2\tilde{U})
+\alpha  \tilde{\eta}^2 \tilde{U}+ \tilde{\eta} \tilde{U}+
\frac{\beta }{2}\left ( \theta ^2-\frac{1}{3}\right ) \tilde{\eta} \Delta \tilde{U}-
\frac{\beta }{3}\tilde{U}(\tilde{U}_{\tilde{x}\tilde{t}}+\tilde{V}_{\tilde{y}\tilde{t}}),\\
&\tilde{q}_{E_y}=\frac{\alpha }{2} (\tilde{V}^3+\tilde{U}^2\tilde{V})
+\alpha  \tilde{\eta}^2 \tilde{V}+ \tilde{\eta} \tilde{V}+
\frac{\beta }{2}\left ( \theta ^2-\frac{1}{3}\right ) \tilde{\eta} \Delta \tilde{V}-
\frac{\beta }{3}\tilde{V}(\tilde{U}_{\tilde{x}\tilde{t}}+\tilde{V}_{\tilde{y}\tilde{t}})\ .
\end{align*}
\end{theorem}
\begin{proof}
We present a full justification only for the first inequality (approximate mass balance). The others follow by similar substitutions.
In fact,
\begin{equation}
 \frac{\partial }{\partial \tilde{t}}\left \{1+\alpha \tilde{\eta } \right \}
 +\frac{\partial }{\partial \tilde{x}}\int_{\tilde{z}=0}^{1+\alpha \tilde{\eta }}\alpha P^{series}_{\tilde{x}}\,d\tilde{z}
 +\frac{\partial }{\partial \tilde{y}}\int_{\tilde{z}=0}^{1+\alpha \tilde{\eta }}\alpha P^{series}_{\tilde{y}}  \,d\tilde{z}=\mathcal{O}(\alpha\beta^2)\ .
\end{equation}
We replace $P^{series}_{\tilde{x}}$ and $P^{series}_{\tilde{y}}$ by expressions \eqref{eq:height333} and \eqref{eq:height444} as follows
\begin{align*}
&\frac{\partial }{\partial \tilde{t}}\left \{ 1+\alpha \tilde{\eta } \right \}  +\frac{\partial }{\partial \tilde{x}}\left\{\alpha\tilde{u}(1+\alpha\tilde{\eta})-\frac{\alpha(1+\alpha\tilde{\eta})^3}{6}\beta\Delta\tilde{u} +\frac{(1+\alpha\tilde{\eta})^5}{120}\alpha\beta^2(\Delta^2\tilde{u})\right\}\\
& +\frac{\partial }{\partial \tilde{y}}\left\{\alpha\tilde{v}(1+\alpha\tilde{\eta})-\frac{\alpha(1+\alpha\tilde{\eta})^3}{6}\beta\Delta\tilde{v}+\frac{(1+\alpha\tilde{\eta})^5}{120}\alpha\beta^2(\Delta^2\tilde{v}) \right\}=\mathcal{O}(\alpha\beta^2).
\end{align*}
Substituting $\tilde{u}$ and $\tilde{v}$ by their expressions in terms of $\tilde{U}$ and $\tilde{V}$ given up to order $\mathcal{O}(\beta^3)$ by \eqref{eq:height333}
and \eqref{eq:height444}, one can write:
$$\frac{\partial }{\partial \tilde{t}}\left\{1+\alpha \tilde{\eta }\right\}+\frac{\partial }{\partial \tilde{x}}\left\{ \tilde{U}(\alpha +\alpha ^2\tilde{\eta })+\frac{\alpha \beta}{2} \left ( \theta ^2
 -\frac{1}{3} \right )\Delta \tilde{U}\right\}
+\frac{\partial }{\partial \tilde{y}}\left\{  \tilde{V}(\alpha +\alpha ^2\tilde{\eta })+\frac{\alpha \beta}{2} \left ( \theta ^2
 -\frac{1}{3} \right )\Delta \tilde{V}\right\}=\Xi.$$
Where,
\begin{align*}
\Xi&=\frac{\partial}{\partial \tilde{x}}\bigg\{\frac{\alpha^2\beta}{6}(\alpha^2\tilde{\eta}^3+3\alpha\tilde{\eta}^2 +3\tilde{\eta}-3\theta^2 \tilde{\eta})\Delta\tilde{U}-\frac{\alpha\beta^2}{120}(\alpha^5\tilde{\eta}^5+5\alpha^4\tilde{\eta}^4+10\alpha^3\tilde{\eta}^3-10\alpha^3\theta^2\tilde{\eta}^3
+10\alpha^2\tilde{\eta}^2\\
&-30\alpha^2\theta^2\tilde{\eta}^2+5\alpha\tilde{\eta}-30\alpha\theta^2\tilde{\eta}+25\alpha\theta^4\tilde{\eta}-10\theta^2+25\theta^4+1)\Delta^2\tilde{U}\bigg\}\\
&+\frac{\partial}{\partial \tilde{y}}\bigg\{\frac{\alpha^2\beta}{6}(\alpha^2\tilde{\eta}^3+3\alpha\tilde{\eta}^2 +3\tilde{\eta}-3\theta^2 \tilde{\eta})\Delta\tilde{V}-\frac{\alpha\beta^2}{120}(\alpha^5\tilde{\eta}^5+5\alpha^4\tilde{\eta}^4+10\alpha^3\tilde{\eta}^3-10\alpha^3\theta^2\tilde{\eta}^3
+10\alpha^2\tilde{\eta}^2\\
&-30\alpha^2\theta^2\tilde{\eta}^2+5\alpha\tilde{\eta}-30\alpha\theta^2\tilde{\eta}+25\alpha\theta^4\tilde{\eta}-10\theta^2+25\theta^4+1)\Delta^2\tilde{V}\bigg\} +\mathcal{O}(\alpha\beta^2).\\
\end{align*}
Since $\tilde{U}$ and $\tilde{V}$ can be expressed in terms of $\tilde{\eta}$ and $\psi$, hence taking $s$ large enough such that $H^s(\R^2) \hookrightarrow L^{\infty}(\R^2)$, we get:
$$\Xi=\mathcal{O}(\alpha\beta^2,\alpha^2\beta).$$
Replacing $P^{series}_{\tilde{x}}$ and $P^{series}_{\tilde{y}}$ by expressions \eqref{eq:height333} and \eqref{eq:height444} once again in 
\begin{align*}
\alpha \frac{\partial }{\partial \tilde{t}}\int_{\tilde{z}=0}^{1+\alpha \tilde{\eta }}P^{series}_{\tilde{x}}\,d\tilde{z}
+\frac{\partial }{\partial \tilde{x}}\int_{\tilde{z}=0}^{1+\alpha \tilde{\eta }} \left \{ \alpha ^2
(P^{series}_{\tilde{x}})^2 
+\alpha \tilde{Q}-(\tilde{z}-1)\right \}\,d\tilde{z}+
\alpha ^2 \frac{\partial }{\partial \tilde{y}} \int_{\tilde{z}=0}^{1+\alpha \tilde{\eta }}P^{series}_{\tilde{x}}P^{series}_{\tilde{y}}\,d\tilde{z}=\mathcal{O}(\alpha\beta,\beta^2)\ ,
\end{align*}
For the approximate momentum balance we give some details on the pressure term,
\begin{align*}
\int_0^{1+\alpha\tilde{\eta}}\tilde{Q}\,d\tilde{z}=-\frac{\partial}{\partial\tilde{x}}\int_0^{1+\alpha\tilde{\eta}}\left\{P^{series}_{\tilde{t}}+\frac{1}{2}\alpha\left((P^{series}_{\tilde{x}})^2+(P^{series}_{\tilde{y}})^2\right)-\frac{1}{2}\frac{\alpha}{\beta}(P^{series}_{\tilde{z}})^2\right\}\,d\tilde{z}    
\end{align*}
Using the expressions of $P^{series}$ and $f$ given by \eqref{eq:exp_P} and \eqref{eq:sbfc} respectively, and then substituting the expressions of $\tilde{u}$ and $\tilde{v}$ in terms of of $\tilde{U}$ and $\tilde{V}$ given by \eqref{eq:height333} and \eqref{eq:height444}, one can write
\begin{align*}
\int_0^{1+\alpha\tilde{\eta}}\tilde{Q}\,d\tilde{z}=\frac{\partial}{\partial\tilde{x}}\left\{\tilde{\eta}+\alpha\tilde{\eta}^2-\frac{\beta}{3}(\tilde{U}_{\tilde{x}\tilde{t}}+\tilde{V}_{\tilde{x}\tilde{t}})+r'\right\}
\end{align*}
with
\begin{align*}
r'&=\frac{\alpha^2\beta}{2}\left(\tilde{\eta}^2+\frac{\alpha}{3}\tilde{\eta}^3\right)(\tilde{U}_{\tilde{x}\tilde{t}}+\tilde{V}_{\tilde{x}\tilde{t}})-\frac{\beta^2}{120}(1+\alpha\tilde{\eta})^5(\Delta\tilde{U}_{\tilde{x}\tilde{t}}+\Delta\tilde{V}_{\tilde{x}\tilde{t}})-\frac{\alpha\beta}{6}(1+\alpha\tilde{\eta})^3\tilde{U}\Delta\tilde{U}+\frac{\alpha\beta^2}{2}\Big(-\frac{\theta^2}{6}(1+\alpha\tilde{\eta})^3\\
&\left((\Delta\tilde{U})^2+(\Delta\tilde{V})^2+\tilde{U}\Delta^2\tilde{U}+\tilde{V}\Delta^2\tilde{V}\right)+\alpha\theta^2\left(\tilde{\eta}^2+\frac{\alpha}{3}\tilde{\eta}^3\right)(\Delta\tilde{U}_{\tilde{x}\tilde{t}}+\Delta\tilde{V}_{\tilde{x}\tilde{t}})+\frac{(1+\alpha\tilde{\eta})^5}{12}\left(\tilde{U}\Delta^2\tilde{U}+\tilde{V}\Delta^2\tilde{V} \right)\\
&+
\frac{(1+\alpha\tilde{\eta})^3}{12}\left((\Delta\tilde{U})^2+(\Delta\tilde{V})^2\right)\Big)-\frac{\alpha\beta}{6}(1+\alpha\tilde{\eta})^3(\tilde{U}_{\tilde{x}}+\tilde{V}_{\tilde{x}})^2+\frac{\alpha\beta^2}{6}\theta^2\left((\Delta{U}_{\tilde{x}})^2+(\Delta{V}_{\tilde{x}})^2+\tilde{U}_{\tilde{x}}\Delta\tilde{V}_{\tilde{x}}+\tilde{V}_{\tilde{x}}\Delta\tilde{U}_{\tilde{x}}\right)\\
&+\frac{\alpha\beta^2}{30}(1+\alpha\tilde{\eta})^5(\tilde{U}_{\tilde{x}}+\tilde{V}_{\tilde{x}})(\Delta\tilde{U}_{\tilde{x}}+\Delta\tilde{V}_{\tilde{x}})-\frac{\beta^2}{6}\theta^2(\Delta\tilde{U}_{\tilde{x}\tilde{t}}+\Delta\tilde{V}_{\tilde{x}\tilde{t}})+\mathcal{O}(\beta^3).
\end{align*}
The other terms are treated similarly as the terms in the approximate mass conservation equation. 
\end{proof}

\section{Numerical validation}\label{sec:numerics}

\setcounter{table}{0}

In this final section, we present a computational study to gain insight into the approximate conservation of mass (\ref{eq:massD}), momentum (\ref{eq:momentum1D})--(\ref{eq:momentum2D}), and energy (\ref{eq:energyD}) in a practical scenario. We focus on the circular expansion of water waves originating from an initial localized source.

For this purpose, we numerically solve the Boussinesq system (\ref{eq:gsystem1})--(\ref{eq:gsystem3}) with $\theta^2=9/11$, using periodic boundary conditions within the domain $D=[-20,20]\times[-20,20]$. Boussinesq systems corresponding to different values of $\theta$ are asymptotically equivalent to the present system and exhibit similar behavior. Our numerical method employs the standard pseudo-spectral approach in combination with the classical four-stage, fourth-order Runge-Kutta method.

We initiate the simulation with the following initial conditions: $\eta(x,y,0)=\exp(-(x^2+y^2)/5)$, $u(x,y,0)=v(x,y,0)=0$, for various values of $\alpha=\beta$ as detailed in Table \ref{tab:errors}. We numerically integrate up to time $T=10$ using a time step of $\Delta t=10^{-4}$ and a spatial step of $\Delta x=0.1$. Throughout the computation, we record values of the discrete quantities (\ref{eq:massD}), (\ref{eq:momentum1D})--(\ref{eq:momentum2D}), and (\ref{eq:energyD}).

To discretize the temporal derivatives in these quantities, we employ forward finite differences, while the spatial derivatives are computed using the fast Fourier transform.
\begin{figure}[h!]
  \centering
  \includegraphics[width=0.8\textwidth]{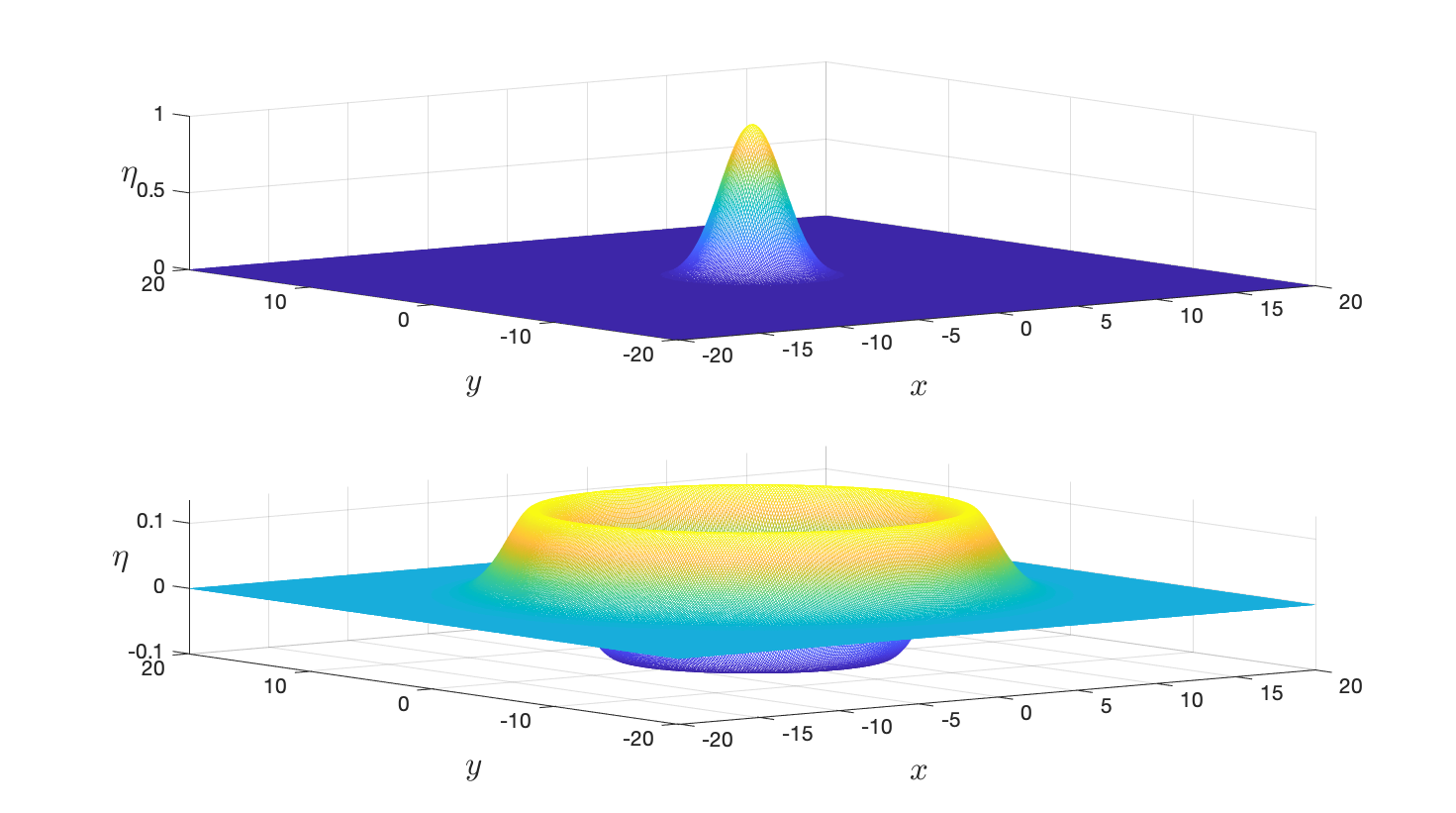}
  \caption{\small Generation of expanding water waves from a heap of water. The waves are propagating in a circular pattern with growing radius.}
  \label{fig:expand}
\end{figure}

\begin{figure}[h!]
  \centering
  \includegraphics[width=0.8\textwidth]{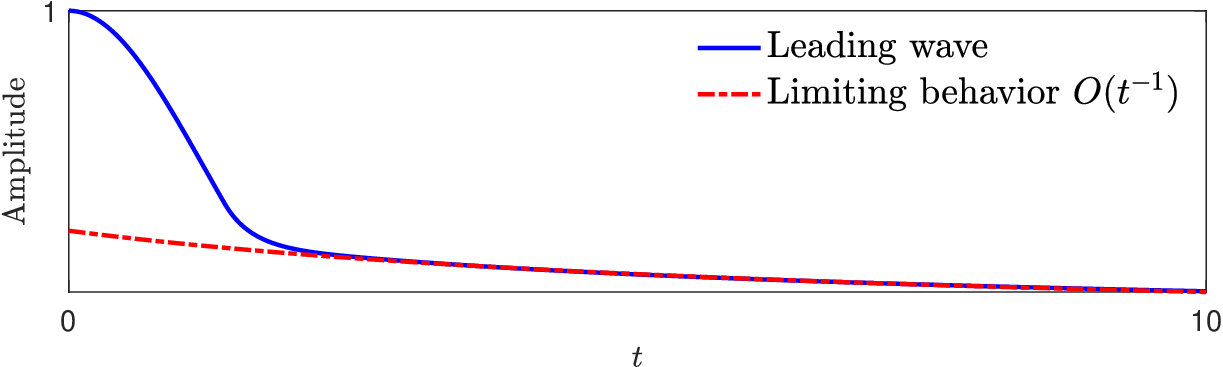}
  \caption{\small Evolution of the amplitude of the leading expanding wave.
The leading wave in the numerical approximation is shown in blue.}
  \label{fig:asympt}
\end{figure}
\begin{table}[ht!]
\begin{center}
\begin{tabular}{cccc}
\hline
$\alpha=\beta$ & Mass & Momentum & Energy \\  
\hline
$0.05$ & $2.21\times 10^{-5}$ & $4.16\times 10^{-3}$ & $3.76\times 10^{-4}$\\
$0.10$ & $1.57\times 10^{-4}$ & $8.22\times 10^{-3}$ & $1.35\times 10^{-3}$\\
$0.15$ & $4.99\times 10^{-4}$ & $1.21\times 10^{-2}$ & $2.84\times 10^{-3}$\\
$0.20$ & $1.12\times 10^{-3}$ & $1.59\times 10^{-2}$ & $4.76\times 10^{-3}$\\
$0.25$ & $2.08\times 10^{-3}$ & $1.96\times 10^{-2}$ & $7.05\times 10^{-3}$\\
$0.30$ & $3.43\times 10^{-3}$ & $2.33\times 10^{-2}$ & $9.67\times 10^{-3}$\\
\hline
\end{tabular}
\end{center}
\caption{\small Residuals of mass, momentum and energy conservation laws evaluated 
using the numerical approximation of a solution of the Boussinesq system.}
\label{tab:errors}
\end{table}
\begin{figure}[h!]
  \centering
  \includegraphics[width=0.8\textwidth]{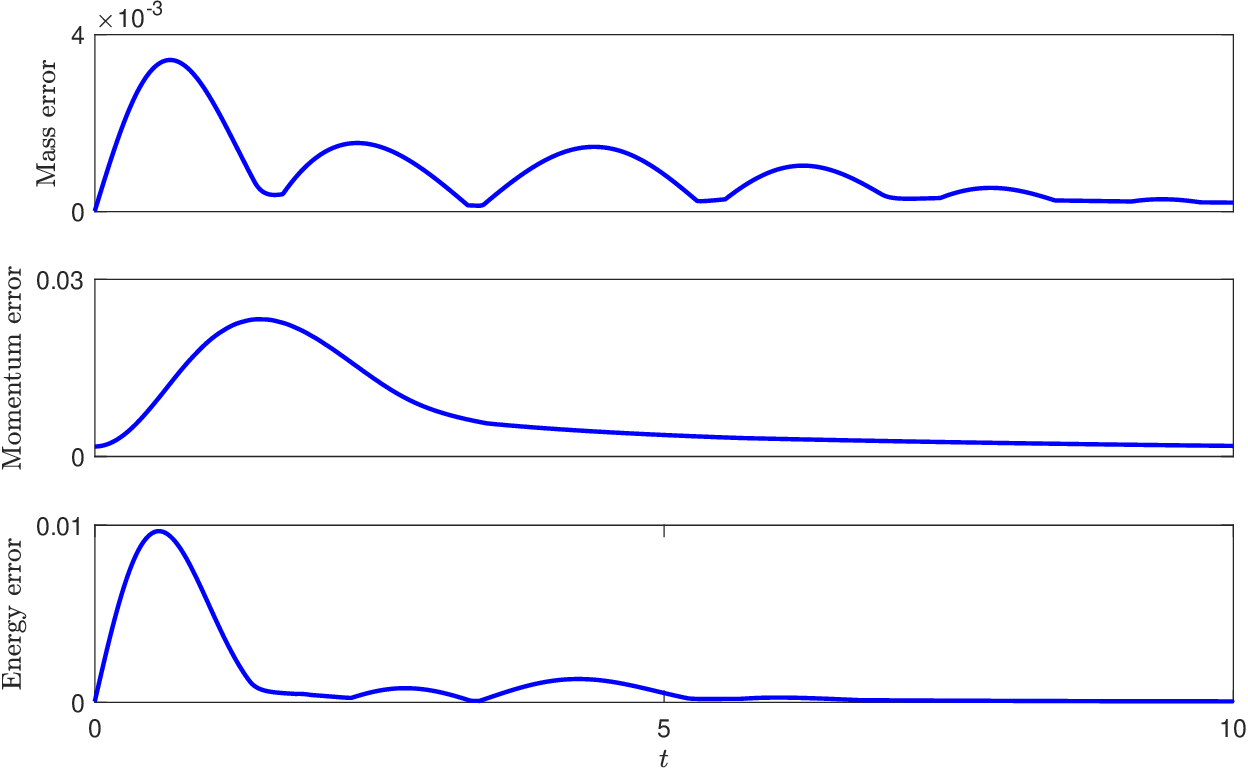}
  \caption{\small Time development of the residuals (absolute errors) of mass, momentum and energy balance laws using the numerical approximation of a solution of the Boussinesq system.}
  \label{fig:errors}
\end{figure} 

The specific initial condition resembles a mound of water, often observed when an object falls into the water or when a tsunami is triggered by changes in the bottom topography (for more details, see \cite{mitsotakis2009boussinesq}).

Figure \ref{fig:expand} illustrates the initial condition and its evolution at $T=10$. The initial condition transforms into a series of expanding waves traveling at varying speeds due to dispersion. Over time, the amplitudes of these waves decrease as a result of dispersion's influence. This reduction in amplitude follows a rate proportional to $t^{-1}$ due to the radial spreading of the two-dimensional waves \cite{Whitham1974}. Figure \ref{fig:asympt} presents the amplitude of the leading expanding wave, and for a detailed derivation of the decrease rate, we direct interested readers to \cite{Whitham1974}.

Subsequently, for the specific initial data presented in Figure \ref{fig:expand}, the computed residuals of the mass, momentum, and energy balance laws initially increase but decay over time. The maximum value attained by each conservation law is listed in Table \ref{tab:errors}. It's worth noting that these maximum values were not reached again, and the residuals appeared to decrease on average. Figure \ref{fig:errors} displays the absolute error patterns for the case where $\alpha=\beta=0.3$; patterns for other cases are quite similar and thus omitted.

A plausible explanation for the decaying of the residuals is that the expanding waves tend to approach linear waves, as explained in the previous section. These numerical results validate the orders of accuracy of the balance laws, particularly in the common scenario of dispersive waves generated from a general initial condition.

\section*{Acknowledgments}
\noindent This research was supported by the Research Council of Norway
under grant numbers 213474/F20 and 239033/F20.
SI would like to thank the Department of Mathematics at the University of Bergen, Norway,
and especially Henrik Kalisch for their kind hospitality when work on this paper was begun.

\bibliographystyle{plain}

\end{document}